\documentclass[12pt,letterpaper]{article}
\usepackage{amssymb,amsfonts,amsmath,amsthm}
\usepackage{adjustbox}
\usepackage{graphicx}
\usepackage{enumerate}
\usepackage[round]{natbib}
\usepackage{url} % not crucial - just used below for the URL 
\usepackage{tikz}
\usepackage{stmaryrd}
\usetikzlibrary{decorations.pathreplacing}

\newcommand{\blind}{1}

\newtheorem{lem}{Lemma}
\newtheorem{cor}{Corollary}
\newtheorem{prp}{Proposition}
\newtheorem{rmk}{Remark}
\newtheorem{thm}{Theorem}
\newtheorem{exm}{Example}
\newtheorem{dfn}{Definition}

\DeclareMathOperator*{\argmax}{arg\,max}
\DeclareMathOperator*{\argmin}{arg\,min}
\newcommand{\E}{\mathrm{E}}
\usepackage[letterpaper,margin=1in]{geometry}
\usepackage{setspace}
\usepackage[final]{microtype}
\usepackage{titlesec}
\titlespacing*{\section}{0pt}{2.0ex plus .4ex minus .2ex}{1.0ex plus .2ex}
\titlespacing*{\subsection}{0pt}{1.5ex plus .3ex minus .2ex}{0.7ex plus .2ex}
\titlespacing*{\subsubsection}{0pt}{1.2ex plus .3ex minus .2ex}{0.5ex plus .2ex}

\usepackage[font=small,labelfont=bf]{caption}
\begin{document}

\def\spacingset#1{\renewcommand{\baselinestretch}%
{#1}\small\normalsize} \spacingset{1}
\setlength{\abovedisplayskip}{5pt}
\setlength{\belowdisplayskip}{5pt}
\setlength{\abovedisplayshortskip}{3pt}
\setlength{\belowdisplayshortskip}{3pt}

%%%%%%%%%%%%%%%%%%%%%%%%%%%%%%%%%%%%%%%%%%%%%%%%%%%%%%%%%%%%%%%%%%%%%%%%%%%%%%

\if1\blind
{
  \title{\bf{Fuzzy Prediction Sets: Conformal Prediction with E-values}
 }
  \author{Nick W. Koning \& Sam van Meer\hspace{.2cm}\\
    Econometric Institute, Erasmus University Rotterdam, the Netherlands}
  \maketitle
} \fi

\if0\blind
{
  \bigskip
  \bigskip
  \bigskip
  \begin{center}
    {\LARGE\bf Fuzzy Prediction Sets: Conformal Prediction with E-values}
\end{center}
  \medskip
} \fi
\vspace{-14mm}
\bigskip

\begin{abstract}
    Prediction sets offer a binary inclusion/exclusion for each element at the same fixed confidence level.
    We generalize to fuzzy prediction sets, which exclude elements at their own data-driven confidence level.
    Our key insight is that a fuzzy prediction set \emph{is} an e-value, capturing precisely what e-values bring to predictive inference.
    Fuzzy prediction sets inherit the merging properties of their e-value, offer richer guarantees to decision-makers.
    We also show in what sense optimal e-values give rise to optimal (fuzzy) prediction sets.
    We apply our results to conformal prediction, deriving optimal fuzzy conformal prediction sets, and characterizing in what sense classical conformal prediction is optimal.
\end{abstract}

% 141 word version. Would be nice to have a 100 word version too
% We propose fuzzy conformal prediction sets that offer a degree of exclusion, beyond the binary inclusion-exclusion of classical prediction sets. We connect fuzzy prediction sets to e-values showing this degree of exclusion corresponds to an exclusion at a different confidence level, capturing what e-values bring to conformal prediction. We show fuzzy prediction sets are predictive distributions with a stronger error guarantee. We derive optimal conformal prediction sets by interpreting the minimization of the expected measure of a prediction set as an optimal testing problem against a particular alternative. We use this to construct optimal (fuzzy) prediction sets, and show in what sense classical conformal prediction is optimal. Finally, we generalize the inheritance of guarantees by subsequent minimax decisions from prediction sets to fuzzy prediction sets. All results generalize beyond the conformal setting to prediction sets for arbitrary models.

\noindent%
\vspace{-.3cm}
    \section{Introduction}
        \subsection{Traditional prediction sets}
            Prediction sets, with conformal prediction sets at the front, have been one of the main success stories in bringing uncertainty quantification to machine learning predictions.
            
            Suppose we have $n$ exchangeable observations $Z^n = (Z_1, \dots, Z_n)$.
            Abstracting away from covariates, the goal of conformal prediction is to use these $n$ observations to construct a prediction set $C_\alpha^{Z^n}$ that will cover the next observation with a probability of at least $1 - \alpha$,
            \begin{align}\label{eq:coverage_guarantee}
                \Pr(Z_{n+1} \in C_\alpha^{Z^n}) \geq 1 - \alpha,
            \end{align}
            assuming $Z^{n+1} = (Z_1, \dots, Z_{n+1})$ is exchangeable.

            To construct such a prediction set, the idea is to plug-in hypothetical values $z$ for $Z_{n+1}$ and test whether $(Z_1, \dots, Z_n, z)$ is exchangeable at level $\alpha$.
            Repeating this exercise for each plug-in value $z$ in the sample space, the prediction set is formed by collecting the values $z$ for which the hypothesis is not rejected at level $\alpha$:
            \begin{align*}
                C_\alpha^{Z^n}
                    = \{z : \textnormal{not reject } (Z_1, \dots, Z_n, z) \textnormal{ is exchangeable}\}.
            \end{align*}
            
            \begin{figure}[htb!]
                \centering
                \includegraphics[width=6.5cm]{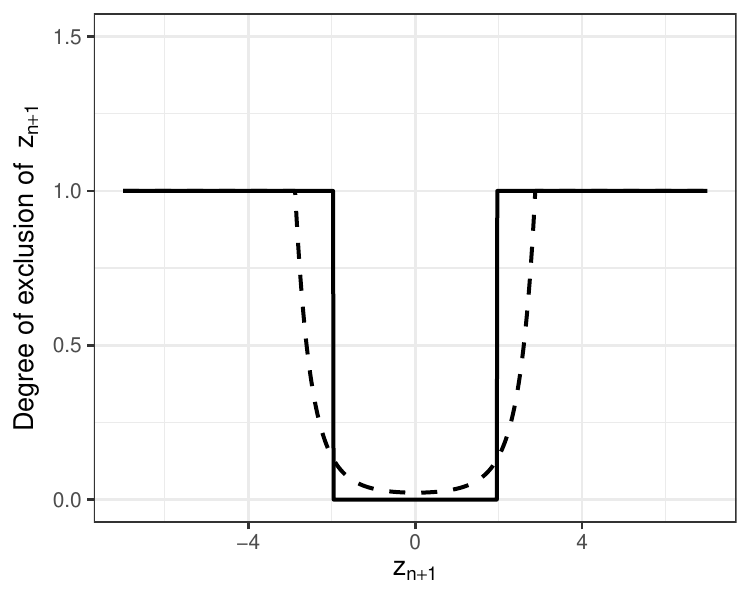}
                \caption{A traditional and fuzzy prediction set. The solid line represents a traditional level $\alpha = 0.05$ prediction set: the points $z_{n+1}$ at which this line is zero form the prediction set, and the points at which it equals $1$ its complement. The dashed line represents a fuzzy $[0, 1]$-valued prediction set, offering a degree of exclusion at each point.}
                \label{fig:fuzzy_confidence_set}
            \end{figure}
            
        \subsection{Fuzzy prediction sets}
            A downside is that a prediction set is binary: a sample point $z$ either falls inside the set or not.
            We overcome this by generalizing to \emph{fuzzy prediction sets}.

            A traditional prediction set may be reinterpreted as a function $C_\alpha^{Z^n} : \mathcal{Z} \to \{0, 1\}$, which outputs an inclusion (0) or exclusion (1) for each point $z \in \mathcal{Z}$.
            A fuzzy prediction set $\widetilde{C}_\alpha^{Z^n} : \mathcal{Z} \to [0, 1]$ generalizes beyond the binary inclusion/exclusion to a \emph{degree of exclusion}.
            We illustrate this in Figure \ref{fig:fuzzy_confidence_set}, where the solid line represents a traditional prediction set and the dashed line a fuzzy prediction set.

            For a fuzzy prediction set, the traditional coverage guarantee \eqref{eq:coverage_guarantee} generalizes to the guarantee that the degree of exclusion of $Z_{n+1}$ is bounded in expectation by $\alpha$:
            \begin{align}\label{eq:fuzzy_guarantee}
                \E[\widetilde{C}_\alpha^{Z^n}(Z_{n+1})] \leq \alpha.
            \end{align}
            For a binary test, this collapses to the classical guarantee \eqref{eq:coverage_guarantee}, as $C_\alpha^{Z^n}(z_{n+1}) = \mathbb{I}\{z_{n+1} \not\in C_\alpha^{Z^n}\}$.

        \subsection{E-values and fuzzy prediction sets}
            We connect fuzzy prediction sets to e-values \citep{grunwald2023safe, vovk2021values, howard2021time, shafer2021testing, koning2025continuoustestingunifyingtests} by merging the degree of exclusion into the significance level, allowing different points $z$ to be excluded at different levels.
            
            To facilitate this, we rescale from $[0, 1]$ to $[0, 1/\alpha]$ by dividing the prediction set by $\alpha$:
            \begin{align*}
                \mathcal{E}^{Z^n}(z)
                    = C_\alpha^{Z^n}(z) / \alpha.
            \end{align*}
            Due to the rescaling, $\mathcal{E}^{Z^n}(z) = 1/\alpha$ now corresponds to a full exclusion at level $\alpha$.
            More generally, we show that the value of a fuzzy prediction set $\mathcal{E}^{Z^n}(z)$ at any point $z$ may be directly interpreted as an exclusion at the data-driven level $1 / \mathcal{E}^{Z^n}(z)$ under an extension of the Type I error to data-dependent levels \citep{grunwald2022beyond, koning2024post}.
            Because of the rescaling, the guarantee \eqref{eq:fuzzy_guarantee} must be reformulated to $\E[\mathcal{E}^{Z^n}(Z_{n+1})] \leq 1$, which also enables $[0, \infty]$-valued fuzzy prediction sets.

            We illustrate fuzzy prediction sets on this \emph{evidence scale} in Figure \ref{fig:fuzzy_confidence_set_rescaled}, which is a rescaled version of Figure \ref{fig:fuzzy_confidence_set} that also includes a $[0, \infty]$-valued fuzzy prediction set.
            Here, the traditional prediction set only offers exclusions at level $1/20 = 0.05$, whereas the fuzzy prediction sets offer exclusions at various significance levels for different points $z$ in the sample space.

            The key insight that underlies this is that a fuzzy prediction set $\mathcal{E}^{Z^n}$, is equivalent to an e-value $\mathcal{E}(Z^n, Z_{n+1}) := \mathcal{E}^{Z^n}(Z_{n+1})$.
            Moreover, we show that the fuzzy prediction set $\mathcal{E}^{Z^n}$ is valid for a model $\mathcal{P}$ if and only if this e-value $\mathcal{E}$ is valid for the hypothesis $\mathcal{P}$.
            This specializes to classical prediction sets and tests, which are $\{0, 1/\alpha\}$-valued e-values.

            We use the equivalence to e-values to show that fuzzy prediction sets inherit their flexible merging properties: they may be averaged under arbitrary dependence, and multiplied under independence.
            Averaging non-fuzzy prediction sets generally produces fuzzy prediction sets.
            Multiplying prediction sets yields their union, and the corresponding confidence level is the product of the individual levels.

            \begin{figure}[htb!]
                \centering
                \includegraphics[width=6.5cm]{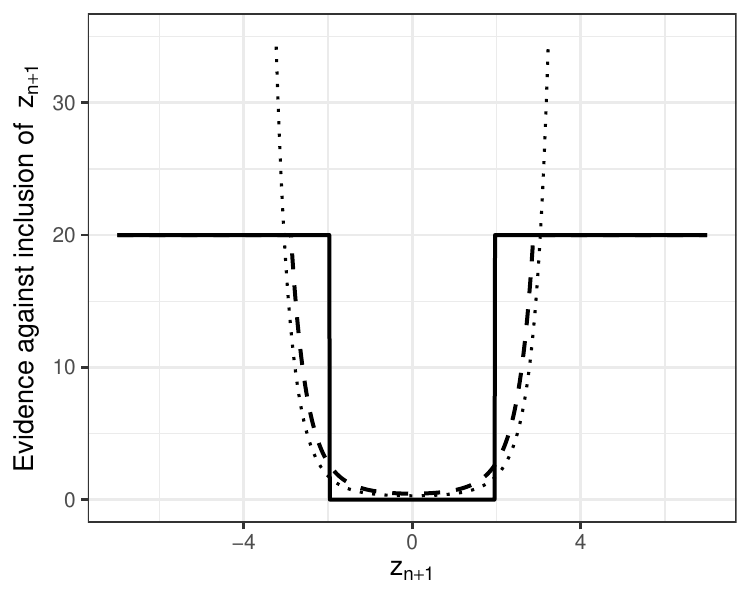}
                \caption{One traditional and two fuzzy prediction sets on the evidence scale. The solid line represents a traditional level $\alpha = 0.05$ prediction set, the dashed line represents a fuzzy $[0, 1/\alpha]$-valued prediction set, and the dotted line a fuzzy $[0, \infty]$-valued prediction set. An evidence value of $\mathcal{E}^{Z^n}(z) = e$ corresponds to an exclusion at the data-driven confidence level $1/e$.}
                \label{fig:fuzzy_confidence_set_rescaled}
            \end{figure}

        \subsection{Optimal (fuzzy) prediction sets}
            The equivalence to an e-value allows us to connect optimal prediction sets to the rich statistical theory on optimal hypothesis testing.
            In particular, we show that a prediction set is of minimal expected size under some measure $\mu$,
            \begin{align*}
                \argmin_{C_\alpha} \E^{Z^n}[\mu(C_\alpha^{Z^n})],
            \end{align*}
            if and only if the corresponding test is most powerful against an alternative that involves $\mu$.
            This most powerful test may be viewed as an e-value that optimizes the expected `Neyman--Pearson utility function' $U^{\textnormal{NP}} : x \mapsto x \wedge 1/\alpha$ \citep{koning2025continuoustestingunifyingtests}. 
            Other utility functions may be used to express different preferences regarding the value of different amounts of evidence.

            We specialize this to conformal prediction in Section \ref{sec:exchangeable_conformal}, leveraging recent results on optimal e-values for exchangeability \citep{koning2023post}.
            This also reveals in what sense classical conformal prediction is optimal, given a choice of conformity score.
            
        \subsection{Decision making, application and covariates}
            To motivate fuzzy prediction sets, we show that they offer richer loss bounds to decision-makers than classical prediction sets.
            For this purpose, we generalize the framework of \citet{kiyani2025decision} for traditional conformal prediction sets.
            One of our approaches connects to \citet{grunwald2023posterior}, and relies on the surprising (to us) observation that a fuzzy confidence set is equivalent to (the reciprocal of) an E-posterior. 
            This shows that the E-posterior is deeply connected to classical hypothesis testing. 
            Furthermore, we unify the related notions of P-certification and E-certification that were recently developed by \citet{andrews2025certified}.
            
            We apply our methodology to character-recognition in Section \ref{sec:application}.
            Here, we showcase the richness of classical prediction sets, and study how different utility functions shape the fuzzy prediction sets.
            In Appendix \ref{sec:covariates}, we extend to covariates.
            
        \subsection{Related literature}
            To the best of our knowledge, fuzzy \emph{confidence} sets were first proposed by \citet{geyer2005fuzzy} as a solution to under-coverage of classical Neyman--Pearson optimal confidence sets for discrete data in small samples.
            We instead advocate for a much wider use, by using a generalization of the Neyman--Pearson framework through e-values.
            A key innovation compared to \citet{geyer2005fuzzy} is that a fuzzy `degree of exclusion at level $\alpha$' is equivalent to an exclusion at a data-driven significance level, which we believe makes fuzzy prediction/confidence sets much more palatable.

            As the literature on conformal prediction and e-values remains fairly small, we attempt to give a complete overview.
            According to \citet{vovk2024conformal}, precursors of conformal prediction \citep{Gammerman1998,Vovk1999} actually relied on e-values (under a different name).
            He attributes their demise in conformal prediction to the fact that prediction sets were more naturally obtained by using p-values.
            This is in line with our finding that e-values are equivalent to \emph{fuzzy} prediction sets, which implies that e-values play no fundamental role in classical (non-fuzzy) prediction sets.
            
            The first explicit connection between modern conformal prediction and e-values appears in the 2020 arXiv version of \citet{vovk2024conformal}, where he exploits the merging properties of e-values to enable cross-conformal prediction by averaging over data splits. 
            Since then, several other works have used e-values as a \emph{pragmatic tool} to construct conformal prediction sets.
            For example, in the context of conformal selection, \cite{lee2024boosting}, \citet{lee2025full}, \citet{lee2025selection} and \citet{nair2025diversifying} compute e-values for (non-)selection events and plug these into the e-BH procedure to control the false discovery rate.
            \citet{gauthier2025values} average conformal e-values over ambiguous ground truths, where observations may carry multiple valid labels simultaneously.
            Moreover, \citet{gauthier2025values} also consider e-values to select a prediction set with a fixed number of elements.
            Our work differs from this line of work, as our intention is to \emph{directly report the e-values as output} in the form of a fuzzy prediction set, rather than using them as an intermediate step to obtain classical prediction sets.
 
            In the early version, \cite{vovk2024conformal} was also first to propose conformal e-values of the form \vspace{-1mm}
            \begin{align}\label{eq:T-based_e-value}
                \frac{T(Z_{n+1})}{\frac{1}{n+1} \sum_{i=1}^{n+1} T(Z_i)},
            \end{align}
            where $T$ is a non-negative statistic.
            This form was later independently rediscovered in the first arXiv version of \citet{wang2022false}, \citet{koning2023post} and \citet{balinsky2024enhancing}, and proven to be the only admissible types of e-values for conformal prediction by \citet{koning2023post}. 
            
            A key issue of the form \eqref{eq:T-based_e-value} is that there is little guidance on how to appropriately select the statistic $T$ and \citet{gauthier2025values} find the outcomes to be highly sensitive to this choice.           
            We turn away from \eqref{eq:T-based_e-value}, and instead introduce a general expected-utility optimality framework to conformal prediction, which enables users to quantify their preference for different degrees of evidence, and allows them to specify a measure $\mu$ that determines the most important regions of the outcome space.
            This can be interpreted as a way to optimally select $T$ in \eqref{eq:T-based_e-value}.
            Our framework nests classical conformal prediction as a special case for a particular choice of the utility function, as well as the setting studied by \citet{vovk2024conformal} for a logarithmic choice of utility function and counting measure $\mu$ in classification problems.
            In our application, we illustrate that log-utility is not always appropriate in the conformal prediction context.
            This is not surprising, as log-utility is generally considered in long-run sequential data settings instead of the single-batch setting common in conformal prediction.

        \section{Relationship between prediction sets and tests}\label{sec:prediction_sets}

            \subsection{Prediction sets}
        
                Let $\mathcal{Z}$ denote the sample space of a single observation, and $\mathcal{Z}^{n+1}$ the joint sample space of $(n+1)$ observations.\footnote{This is easily generalized to more general joint sample spaces, but we stick to copies of the same sample space for the sake of exposition.}
                A prediction set is a set $C_\alpha^{Z^{n}}$ with the level $\alpha > 0$ coverage guarantee
                \begin{align}\label{ineq:cover}
                    P(Z_{n+1} \in C_\alpha^{Z^{n}}) \geq 1 - \alpha, \textnormal{ for every } P \in \mathcal{P},
                \end{align}
                where $\mathcal{P}$ is some \emph{model} (collection of distributions) on $\mathcal{Z}^{n+1}$.
                This means the prediction set $C_\alpha^{Z^{n}}$ will cover $Z_{n+1}$ with at least $1 - \alpha$ probability.
                We say that such a prediction set is \emph{valid} for the model $\mathcal{P}$.
                Notice that the probability in \eqref{ineq:cover} is over both $Z^n$ and $Z_{n+1}$, so that the coverage guarantee is marginal over the entire tuple $Z^{n+1}$.
                
                Conformal prediction is the special case where the model $\mathcal{P}$ is the collection $\mathcal{P}^{\textnormal{exch.}}$ of exchangeable distributions on $\mathcal{Z}^{n+1}$, but we delay specializing to exchangeability to Section \ref{sec:exchangeable_conformal}.
                Moreover, we postpone the inclusion of covariates to Appendix \ref{sec:covariates} as their inclusion considerably clutters the notation and is not relevant for the discussion here.
                % Appendix \ref{sec:examples} provides numerous examples of our methodology for Gaussian prediction sets.

            \subsection{Slicing a prediction set for $Z^{n+1}$}\label{sec:slicing}
                \begin{figure}
                    \centering
                    \includegraphics[width=0.6\linewidth]{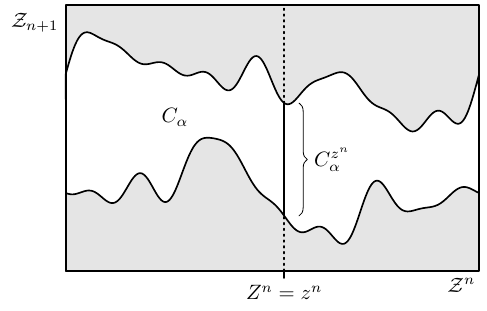}
                    \caption{Abstract illustration of the sample space $\mathcal{Z}^{n+1}$, where the horizontal dimension represents $\mathcal{Z}^n$ and the vertical dimension $\mathcal{Z}_{n+1}$.
                    The prediction set $C_\alpha$ for $Z^{n+1}$ is shaded white, and its complement $\mathcal{Z}^{n+1} \setminus C_\alpha$ grey.
                    The prediction set $C_\alpha^{Z^n}$ for $Z_{n+1}$ appears as a $Z^n$-dependent random slice of $C_\alpha$.}
                    \label{fig:rejection_region}
                \end{figure}

                The first key insight is that the construction of a prediction set $C_\alpha^{Z^n}$ for $Z_{n+1}$ is equivalent to constructing a prediction set $C_\alpha$ for $Z^{n+1}$.
                Indeed, the prediction set $C_\alpha^{z^n}$ can be interpreted as a `slice' of a prediction set $C_\alpha$:
                \begin{align}\label{eq:slicing}
                    C_\alpha^{z^n}
                        = \{z_{n+1} \in \mathcal{Z} : (z^n, z_{n+1}) \in C_\alpha\}.
                \end{align}                
                % This slice-representation is without loss of generality, as a confidence set $C_\alpha$ may be recovered from the slices $C_\alpha^{z^n}$:
                % \begin{align*}
                %     C_\alpha 
                %         = \{z^{n+1} \in \mathcal{Z}^{n+1} : z_{n+1} \in C_\alpha^{z^n}\},
                % \end{align*}
                % and specifying $C_\alpha^{Z^n}$ requires specifying its for each realization of $z^n$.
                
                We illustrate this in Figure \ref{fig:rejection_region}, where the union of the slices $\{C_\alpha^{z^n}\}_{z^n \in \mathcal{Z}^n}$ forms the prediction set $C_\alpha$.
                Here, we stress that the prediction set $C_\alpha$ is deterministic: the randomness in $C_\alpha^{Z^n}$ is purely due to the random realization of the point $Z^n$ at which it is `sliced'.

                Validity of $C_\alpha$ as a prediction set for $Z^{n+1}$ under the model $\mathcal{P}$ is equivalent to the validity of $C_\alpha^{Z^n}$ in the sense of \eqref{ineq:cover}.
                Indeed, by \eqref{eq:slicing} we have $\mathbb{I}\{Z^{n+1} \in C_\alpha\} = \mathbb{I}\{Z_{n+1} \in C_\alpha^{Z^{n}}\}$, and so
                \begin{align*}
                    P(Z^{n+1} \in C_\alpha)
                        = P(Z_{n+1} \in C_\alpha^{Z^{n}}),
                \end{align*}
                for any $P$.
                As a consequence, we may focus on constructing prediction sets $C_\alpha$ for $Z^{n+1}$.

                \begin{rmk}[Slicing on observables]
                    These arguments may be extended to slicing based on any part or statistic of $Z^{n+1}$ that is observed.
                    We illustrate this in the covariates setting described in Appendix \ref{sec:covariates} where $Z_{n+1} = (X_{n+1}, Y_{n+1})$ and we observe $X_{n+1}$ alongside $Z^n$, to construct a prediction set for $Y_{n+1}$.
                    There, we still construct a prediction set $C_\alpha$ for $Z^{n+1}$ but then slice it at $(Z^n, X_{n+1})$.
                \end{rmk}

            \subsection{Prediction as a testing problem}
                A second insight is that a prediction set $C_\alpha$ is equivalent to a single hypothesis test.
                This is a powerful observation, because the construction of hypothesis tests is one of the most deeply studied topics in statistics.
                
                Recall from the introduction that we may view a prediction set $C_\alpha$ for $\mathcal{P}$ as a map $C_\alpha : \mathcal{Z}^{n+1} \to \{0, 1\}$, which returns an inclusion (0) or exclusion (1) for each point $z^{n+1} \in \mathcal{Z}^{n+1}$.
                We choose to interpret this map $C_\alpha$ as a test, where 0 represents a non-rejection and 1 a rejection.
                Coupling this to Figure \ref{fig:rejection_region}, the white region corresponds to the non-rejection region of this test and the grey region to its rejection region.
                To avoid using two different names for the same object, we use the notation $C_\alpha$ both for the set and test representation.
                
                The prediction set $C_\alpha$ for $Z^{n+1}$ (and by extension the prediction set $C_\alpha^{Z^n}$ for $Z_{n+1}$) is valid for the model $\mathcal{P}$ if and only if it is a valid test for the `hypothesis' $\mathcal{P}$, since
                \begin{align*}
                    P(Z_{n+1} \not\in C_\alpha^{Z^n})
                        = P(Z^{n+1} \not\in C_\alpha)
                        = \E^P[C_\alpha],
                \end{align*}
                for every $P \in \mathcal{P}$.
                A consequence is that \emph{any} test that is valid for $\mathcal{P}$ automatically produces a valid prediction set for the model $\mathcal{P}$, and vice-versa.
                
                Given the importance of this observation, we capture it in Theorem \ref{thm:confidence_set_test_equivalence} for future reference.
                
                \begin{thm}\label{thm:confidence_set_test_equivalence}
                    A prediction set $C_\alpha^{Z^n}$ is equivalent to the test $C_\alpha(z^{n+1})
                            := \mathbb{I}\{z^{n+1} \not\in C_\alpha^{z^n}\}$.
                    The prediction set $C_\alpha^{Z^n}$ is valid at level $\alpha$ for the model $\mathcal{P}$ if and only if this test $C_\alpha$ is valid at level $\alpha$ for the hypothesis $\mathcal{P}$.
                \end{thm}

                \begin{rmk}[Connection to classical testing]
                    At first glance the connection to testing may seem superficial, but we argue it is not.
                    Recall that in classical hypothesis testing we hypothesize a model $\mathcal{P}$ (often called ``the hypothesis'') and subsequently observe data.
                    As we observe the data, we must reject the hypothesized model if the data and model are incompatible.
                    In the construction of a prediction sets the roles are reversed: we fix a model $\mathcal{P}$ as the ground truth and hypothesize data $Z^{n+1} = z^{n+1}$.
                    If the two are incompatible, we are forced to reject the hypothesized data.
                    Extending this argument: if a part $Z^n = z^n$ of $Z^{n+1}$ is also observed alongside the model $\mathcal{P}$, we must reject the hypothesized part: $Z_{n+1} = z_{n+1}$.
                \end{rmk}
                
            \subsection{Optimal prediction sets and power}\label{sec:confidence_width_power}
                By Theorem \ref{thm:confidence_set_test_equivalence}, constructing a prediction set is equivalent to constructing a test.
                It remains to select a test that yields a `good' or even optimal prediction set in some appropriate sense.
                One may hope that an optimal prediction set corresponds to a test that is most powerful against some alternative.
                We show this is the case, by deriving the alternative hypothesis under which the test should be most powerful for a prediction set to be `as small as possible'.

                To formalize what we mean with `as small as possible', let $\mu$ denote some unsigned measure on $\mathcal{Z}$, selected to measure the size of $C_\alpha^{Z^n}$.
                For the sake of exposition, let us \emph{temporarily} assume we have access to the true distribution $P_*^{Z^n}$ of $Z^n$.
                It then seems natural to formalize our problem as minimizing the expected $\mu$-size of $C_\alpha^{Z^n}$ under $P_*^{Z^n}$:
                \begin{align*}
                    \argmin_{C_\alpha}\E^{P_*^{Z^n}}[\mu(C_\alpha^{Z^{n}})],
                \end{align*}
                over valid prediction sets $C_\alpha$.
                For example, if $\mathcal{Z}$ is finite then a natural choice for $\mu$ may be the counting measure, which leads to the minimization of the number of elements in our prediction set: \vspace{-6mm}
                \begin{align*}
                \argmin_{C_\alpha}\E^{P_*^{Z^n}}[|C_\alpha^{Z^{n}}|].
                \end{align*}

                In Theorem \ref{thm:most_powerful}, we show that $C_\alpha$  minimizes the expected size of $C_\alpha^{Z^n}$ if and only if its test-representation is most powerful against the alternative $Q = P_*^{Z^n} \otimes \mu$.
                In the result, we even show that the choice of $\mu$ may depend on $Z^n$. The proof of this result and all other omitted proofs can be found in Appendix \ref{sec:proofs}.
                In Section~\ref{sec:examples}, we include four examples where we apply this result in a Gaussian location setting.
                
                \begin{thm}\label{thm:most_powerful}
                    Let $\mu^{|Z^n}$ be a probability kernel dependent on $Z^n$.
                    Then $C_\alpha$ minimizes
                    \begin{align*}
                        \E^{P_*^{Z^n}}[\mu^{|Z^n}(C_\alpha^{Z^n})],
                    \end{align*}
                    if and only if the test $C_\alpha$ is most powerful against the alternative $Q = P_*^{Z^n} \otimes \mu^{|Z^n}$.
                \end{thm}

            \subsection{Handling unknown $P_*^{Z^n}$}
                In practice, the true distribution $P_*^{Z^n}$ is generally unknown.
                Fortunately, we can just plug-in an educated guess or estimator $\widehat{P}^{Z^n}$ for $P_*^{Z^n}$, based on a separate dataset.
                The resulting prediction set will still be valid for $\mathcal{P}$.
                The only consequence is that it is optimal for the estimator $\widehat{P}^{Z^n}$, instead of for the true distribution $P_*^{Z^n}$.                
                This mirrors the classical problem in hypothesis testing that if we do not `know' the true alternative distribution in case of a composite alternative, then we can rarely hope to construct a test that is most powerful against it, except for extremely rare settings where there exist uniformly most powerful tests.
                
                In case there is no structure available, a general-purpose tool is to construct a valid test against several (or all) possible guesses of $P_*^{Z^n}$, and choose our test as some (probability)-weighted average over these tests.
                However, such an average would usually take value in $[0, 1]$, so that this is an example of a fuzzy test.
                Another strategy is to construct an optimal test against such a weighted average over different resulting choices of $Q$.

                % We suspect that the only setting in which this problem completely vanishes is when $Z^n$ is sufficient for $\mathcal{P}$.
                % Abstractly, sufficiency of $Z^n$ means that the same conditional distribution $P_{\dag}^{|Z^n}$ is shared by all $P \in \mathcal{P}$.
                % To see why this is helpful, notice that if we had access to the true data-generating distribution $P_* \in \mathcal{P}$, then the Neyman--Pearson lemma informs us that the the optimal test is the likelihood ratio test based on the likelihood ratio statistic
                % \begin{align*}
                %     \textnormal{LR}^*(z^{n+1})
                %         = \frac{d(P_*^{Z^n} \otimes \mu^{|Z^n})}{d(P_*^{Z^n} \otimes P_*^{|Z^n})}(z^{n+1})
                %         = \frac{d\mu^{|Z^n}}{dP_*^{|Z^n}}(z_{n+1}).
                % \end{align*}
                % In general, this conditional likelihood ratio of course cannot be used, as $P_*^{|Z^n}$ may be unknown.
                % However, if $Z^n$ is sufficient then $P_*^{|Z^n} = P_{\dag}^{|Z^n}$ is known, because the same conditional distribution is shared by all $P \in \mathcal{P}$, including $P_*$.

            \subsection{Examples: Gaussian location}\label{sec:examples}
                In this section, we illustrate our framework in Gaussian location models.
                For natural choices of the kernel $\mu^{|Z^n}$ we recover familiar or expected prediction sets.
                At the same time, we find that the classical two-sided $z$-test produces a strange prediction set, matching a kernel $\mu^{|Z^n}$ that emphasizes a peculiar part of the sample space.
                                
                \begin{exm}[i.i.d. simple Gaussian]\label{exm:simple_gaussian_NP}
                    When $\mathcal{P}$ is a singleton, and both $Z_{n+1}$ and $\mu$ do not depend on $Z^n$,  there is nothing to learn from $Z_n$, so the optimal prediction set is deterministic.

                    To illustrate this, let $\mathcal{Z} = \mathbb{R}$, and $\mathcal{P} = \{\mathcal{N}(\delta 1_{n+1}, \sigma^2 I_{n+1})\}$, with known $\delta \in \mathbb{R}$ and $\sigma > 0$.
                    Following Theorem \ref{thm:most_powerful}, we choose a measure $\mu^{|Z^n}$ that expresses what we mean by a small prediction set.
                    A natural choice is the Lebesgue measure $\mu^{|Z^n} = \lambda$.
                    However, as $\lambda$ is not finite, we replace it by $\mathcal{N}(\delta,\tau^2)$ for $\tau>\sigma$.
                    This is of no consequence here, because the likelihood ratio $\frac{d\mathcal{N}(\delta, \tau^2)}{d\mathcal{N}(\delta, \sigma^2)}$ is a monotone transformation of $\frac{d\lambda}{d\mathcal{N}(\delta, \sigma^2)}$, so that both yield the same optimal test by invariance of the Neyman-Pearson test to monotone transformations.

                    For this choice of $\mu^{|Z^n}$, we obtain the alternative $Q = \mathcal{N}(\delta 1_{n}, \sigma^2 I_{n})\times \mathcal{N}(\delta,\tau^2)$.
                    Since the first $n$ marginals of $Q$ equal those under $P\in \mathcal{P}$, the likelihood ratio reduces to $\frac{d\mathcal{N}(\delta, \tau^2)}{d\mathcal{N}(\delta, \sigma^2)}(z_{n+1})$, which is increasing in $|z_{n+1}-\delta|$.
                    The most powerful test rejects when $|z_{n+1}-\delta|>\sigma c_{1-\alpha/2}$, where $ c_{1-\alpha/2}$ denotes the standard Gaussian's $(1-\alpha/2)$-quantile.
                    Inverting this test gives the prediction set $[\delta-\sigma c_{1-\alpha/2},\delta+\sigma c_{1-\alpha/2}]$: the standard Gaussian prediction interval.
                    
                \end{exm}
                \begin{exm}[i.i.d. Composite Gaussian]\label{exm:composite_gaussian_NP}
                    If $\mathcal{P}$ is composite, the first $n$ observations help to narrow down the true measure $P_*^{Z^{n+1}} \in \mathcal{P}$, so the optimal prediction set $C_\alpha^{Z^n}$ depends on $Z^n$.
                    
                    To illustrate this, let $\mathcal{Z} = \mathbb{R}$ and $\mathcal{P} = \{\mathcal{N}(\delta 1_{n+1}, \sigma^2 I_{n+1}) : \delta \in \mathbb{R}\}$ with known $\sigma > 0$.
                    We again minimize the expected Lebesgue measure and, as in Example~\ref{exm:simple_gaussian_NP}, replace $\lambda$ by a Gaussian with larger variance.
                    Since $P_*^{Z^n}$ is now unknown, the alternative becomes composite: $\mathcal{Q} = \{\mathcal{N}(\delta 1_{n}, \sigma^2 I_{n}) \times \mathcal{N}(\delta, \tau^2) : \delta \in \mathbb{R}\}$, $\tau > \sigma$.
                    Both $\mathcal{P}$ and $\mathcal{Q}$ are invariant under $x \mapsto x + c1_{n+1}$, so by the Hunt--Stein theorem we may restrict to tests of the maximal invariant $(Z^n - \bar{Z}_n, Z_{n+1} - \bar{Z}_n)$.
                    The first component has the same distribution under $\mathcal{P}$ and $\mathcal{Q}$, so the likelihood ratio depends only on $B = Z_{n+1} - \bar{Z}_n$.
   
                    Under $\mathcal{P}$, $B \sim \mathcal{N}(0, \sigma^2(1 + 1/n))$, and under $\mathcal{Q}$, $B \sim \mathcal{N}(0, \tau^2(1 + 1/n))$, so the likelihood ratio is increasing in $|B|$ and the most powerful test rejects if $|B| > c_{1 - \alpha/2}\,\sigma\sqrt{1 + 1/n}$, so
                    \begin{align*}
                       C_\alpha^{Z^n} = [\bar{Z}_{n} - c_{1 - \alpha/2} \sigma \sqrt{1 + 1/n},\; \bar{Z}_{n} + c_{1 - \alpha/2} \sigma \sqrt{1 + 1/n}].
                    \end{align*}
                    Compared to Example~\ref{exm:simple_gaussian_NP}, the center $\delta$ is replaced by the estimator $\bar{Z}_n$ and the resulting estimation uncertainty inflates the width by a factor $\sqrt{1 + 1 / n}$.
                \end{exm}

                \begin{exm}[Autoregressive Gaussian]\label{exm:simple_gaussian_AR_NP}
                    In the previous examples, the kernel $\mu^{|Z^n}$ was independent of $Z^n$.
                    This is usually not the case if $Z_{n+1}$ and $Z^n$ are dependent.
                    
                    Suppose $\mathcal{P} = \{\mathcal{N}(\delta 1_{n+1}, \Sigma)\}$, where $\Sigma_{i,j} = \sigma^2\rho^{|i-j|}$ corresponds to an AR(1) model with $\rho \in (-1,1)$ and $\sigma > 0$.
                    Here, it is natural to choose the kernel $\mu^{|Z^n} = \mathcal{N}(\delta_n, \tau^2)$ centered at the conditional mean $\delta_n = \delta + \rho(z_n - \delta)$ of $Z_{n+1}$ given $Z^n$, with $\tau > \sigma\sqrt{1 - \rho^2}$.
                    As in Example~\ref{exm:simple_gaussian_NP}, the likelihood ratio depends only on $z_{n+1}$ and is increasing in $|z_{n+1} - \delta_n|$, so the most powerful test rejects when $|z_{n+1} - \delta_n| > \sigma\sqrt{1 - \rho^2}\, c_{1 - \alpha/2}$, yielding the prediction set $[\delta_n - \sigma\sqrt{1 - \rho^2}\, c_{1 - \alpha/2},\; \delta_n + \sigma\sqrt{1 - \rho^2}\, c_{1 - \alpha/2}]$: the standard AR(1) prediction interval.
                    This illustrates the role of the kernel $\mu^{|Z^n}$: by centering it at the conditional mean, the prediction set adapts to the dependence structure.
            \end{exm}
            
            \begin{exm}[Prediction set based on two-sided $z$-test]\label{exm:2-sided_z-test}
                While every test corresponds to a prediction set, classical tests may not produce desirable prediction sets.

                To illustrate this, we consider the prediction set that corresponds to the classical two-sided $z$-test for $\mathcal{P} = \{\mathcal{N}(\delta 1_{n+1}, \sigma^2 I_{n+1})\}$, which rejects if $|\bar{Z}_{n+1} - \delta| > \sigma c_{1-\alpha/2} / \sqrt{n+1}$.
                Inverting in $z_{n+1}$ produces the prediction set
                \begin{align*}
                        [(n+1)\delta - n\bar{Z}_n - \sqrt{n+1}\, \sigma c_{1-\alpha/2},(n+1)\delta - n\bar{Z}_n + \sqrt{n+1}\, \sigma c_{1-\alpha/2} ].
                \end{align*}
                Compared to Example~\ref{exm:simple_gaussian_NP}, this set is centered at $(n+1)\delta - n\bar{Z}_n$ instead of $\delta$ and it is wider by a factor $\sqrt{n + 1}$.
                Inverting Theorem~\ref{thm:most_powerful}, we can deduce that it is optimal for the kernel $\mu^{|Z^n} = \mathcal{N}(\delta + n(\bar{Z}_n - \delta),\, 2\sigma^2)$, which emphasizes a region far from $\delta$.
            \end{exm}
             
        \section{Fuzzy prediction sets and e-values}\label{sec:fuzzy_conformal}
            Fuzzy prediction sets move beyond the binary exclusion $Z_{n+1} \not\in C_\alpha^{Z^n}$ at a prespecified level $\alpha$ offered by a classical prediction set $C_\alpha^{Z^n}$.
            The underlying idea is to generalize from tests to e-values, where we follow the perspective presented in \citet{koning2025continuoustestingunifyingtests}.

            \subsection{From tests to e-values}\label{sec:fuzzy_is_different_alpha}
                Without loss of generality, the first step is to incorporate the level $\alpha$ into the decision space: we redefine a level $\alpha$ prediction set (test) to be $\{0, 1/\alpha\}$-valued instead of $\{0, 1\}$-valued
                \begin{align*}
                    \mathcal{E}_\alpha^{Z^n} : \mathcal{Z} \to \{0, 1/\alpha\}.
                \end{align*}
                Here, we switch the notation from $C$ to $\mathcal{E}$ to emphasize the different scale.
                Moreover, due to the rescaling, the prediction set is now valid if $\E^P[\mathcal{E}_\alpha^{Z^n}(Z_{n+1})] \leq 1$.

                The key insight is that incorporating the level into the decision space enables us to go beyond the binary exclusion by incorporating more options into the decision space: 
                \begin{align*}
                    \mathcal{E}^{z^n} : \mathcal{Z} \to \{0, 1/\alpha_1, 1/\alpha_2, \dots\},
                \end{align*} 
                $\alpha_1, \alpha_2 > 0$.
                Allowing every value of $\alpha > 0$ yields the most general form $\mathcal{E}^{z^n} : \mathcal{Z} \to [0, \infty]$.
                The interpretation is that each point $z^{n+1}$ is excluded at its own data-driven significance level $\widetilde{\alpha}(z^{n+1}) = 1 / \mathcal{E}^{z^n}(z_{n+1})$.

                \begin{dfn}[Fuzzy prediction set]
                    A fuzzy prediction set for $Z_{n+1}$ is a measurable map $\mathcal{E} : \mathcal{Z}^{n+1} \to [0, \infty]$.
                    It is valid for $\mathcal{P}$ if $\E^P[\mathcal{E}^{Z^n}(Z_{n+1})] \leq 1$, for every $P \in \mathcal{P}$.
                \end{dfn}

                \begin{rmk}[Randomization]\label{rmk:randomized}
                    If desired, a binary prediction set at a prespecified level $\alpha$ may be retrieved from a fuzzy prediction set through randomization.
                    In particular, we may independently draw $U \sim \textnormal{Unif}[0, 1]$, and construct the (randomized) prediction set $\mathcal{E}_\alpha^{Z^n}
                            = \frac{1}{\alpha} \mathbb{I}\left\{\mathcal{E}^{Z^n} \geq U / \alpha\right\}$.
                    This randomized prediction set then still satisfies the classical level $\alpha$ coverage guarantee, where the coverage probability includes the randomization.
                    We do not recommend such a randomized procedure unless a binary prediction set is practically necessary.
                \end{rmk}

            \subsection{Data-dependent Type-I error}
                Since a fuzzy prediction set excludes points at a data-driven confidence level, we must be careful in their interpretation as the classical Type-I coverage guarantee only concerns data-independent levels.
                We must therefore extend the classical Type-I error coverage guarantee to data-dependent levels \citep{koning2024post, grunwald2022beyond}.

                To prepare for the interpretation, we collect the points $z \in \mathcal{Z}$ not rejected at level $\alpha$:
                \begin{align}\label{eq:sublevel_sets}
                    \overline{C}_\alpha^{z^n}
                        = \{z \in \mathcal{Z} : \mathcal{E}^{z^n}(z) < 1/\alpha\}.
                \end{align}
                Repeating this exercise for every $\alpha > 0$ yields a collection of prediction sets $(\overline{C}_\alpha^{z^n})_{\alpha > 0}$.
                This collection $(\overline{C}_\alpha^{z^n})_{\alpha > 0}$ is equivalent to $\mathcal{E}^{z^n}$, since $\mathcal{E}^{z^n}(z) = \sup\{1/\alpha :  z \not\in \overline{C}_\alpha^{z^n}\}$.

                Theorem \ref{thm:post-hoc} shows that this collection of prediction sets $( \overline{C}_\alpha)_{\alpha > 0}$ is \emph{post-hoc} valid; marginally over all data-dependent significance levels $\widetilde{\alpha}$.
                The proof follows from Theorem 2 in \citet{koning2024post}, applied to the test $\overline{C}_\alpha^{z^n} = \mathbb{I}\{z_{n+1} \not\in \overline{C}_\alpha^{z^n}\}$.
                The subsequent Lemma \ref{lem:smallest-level} follows immediately from \eqref{eq:sublevel_sets}.
    
                \begin{thm}\label{thm:post-hoc}
                    The collection of sublevel sets $(\overline{C}_\alpha^{Z^n})_{\alpha > 0}$ of a fuzzy prediction set $\mathcal{E}^{Z^n}$ satisfies
                    \begin{align}\label{dfn:post-hoc_valid}
                        \E_{\widetilde{\alpha}}^P\left[\frac{P(Z_{n+1} \not\in \overline{C}_{\widetilde{\alpha}}^{Z^n} \mid \widetilde{\alpha})}{\widetilde{\alpha}}\right] \leq 1, \textnormal{ for every } P \in \mathcal{P},
                    \end{align}
                    for every data-dependent significance level $\widetilde{\alpha}$ if and only if $\mathcal{E}^{Z^n}$ is valid.
                \end{thm}
    
                \begin{lem}\label{lem:smallest-level}
                    The smallest data-dependent level $\widetilde{\alpha}$ for which $z$ is excluded from $\overline{C}_{\widetilde{\alpha}}^{Z^n}$ equals $\widetilde{\alpha} = 1/\mathcal{E}^{Z^n}(z)$.
                \end{lem}

                \begin{rmk}[Interpretation guarantee]
                    The guarantee \eqref{dfn:post-hoc_valid} should not be misinterpreted as the conditional guarantee $P(Z_{n+1} \not\in C_{\widetilde{\alpha}}^{Z^n} \mid \widetilde{\alpha}) / \widetilde{\alpha} \leq 1$, for every realization of $\widetilde{\alpha}$, and every $\widetilde{\alpha}$.
                    Indeed, \eqref{dfn:post-hoc_valid} only offers this in expectation over $\widetilde{\alpha}$.
                    At the same time, \eqref{dfn:post-hoc_valid} is stronger than $P(Z_{n+1} \not\in C_{\widetilde{\alpha}}^{Z^n})  \leq \E^P[\widetilde{\alpha}]$; see Example 8 in \citet{koning2024post}.
                \end{rmk}

                \begin{rmk}
                    As a traditional (non-fuzzy) level $\alpha$ prediction set may be seen as a special binary fuzzy prediction set, we may technically also use data-dependent levels with traditional prediction sets.
                    However, we then find that if we select $\widetilde{\alpha} > \alpha$ this yields the same prediction set as $\alpha$ but with a weaker guarantee $\widetilde{\alpha}$, and if $\widetilde{\alpha} < \alpha$ then the resulting prediction set is the entire sample space.
                    In this sense, a non-fuzzy level $\alpha$ prediction set effectively forces the choice $\widetilde{\alpha} = \alpha$ as the only sensible option.
                \end{rmk}

                \begin{rmk}[Relationship to predictive distributions]
                    In the conformal setting, nested collections of valid prediction sets are sometimes used to construct predictive distributions for $Z_{n+1}$ \citep{vovk2017nonparametric, gneiting2014probabilistic}.
                    Such a predictive distribution assigns a calibrated probability to every event, but this calibration breaks down for events chosen based on the realized predictive distribution.
                    A valid fuzzy prediction set does not yield a probability measure, but its sublevel sets form a nested collection of prediction sets that retain a validity guarantee when chosen post-hoc, based on the realized fuzzy prediction set.
                \end{rmk}

            \subsection{A fuzzy prediction set is an e-value}
                We now present Theorem \ref{thm:equivalence_fuzzy_confidence_e-value}, which shows that a fuzzy prediction set $\mathcal{E}^{Z^n}$ for $Z_{n+1}$ under model $\mathcal{P}$ is equivalent to an e-value $\mathcal{E}$ defined as $\mathcal{E}(z^{n+1}) = \mathcal{E}^{z^n}(z_{n+1})$ for the hypothesis $\mathcal{P}$.
                This means that to construct a fuzzy prediction set $\mathcal{E}^{Z^n}$, it suffices to construct an e-value $\mathcal{E}$ for $\mathcal{P}$.
                Theorem \ref{thm:equivalence_fuzzy_confidence_e-value} generalizes Theorem \ref{thm:confidence_set_test_equivalence} from non-fuzzy to fuzzy prediction sets.
                
                \begin{thm}\label{thm:equivalence_fuzzy_confidence_e-value}
                    A fuzzy prediction set $\mathcal{E}^{Z^n} : \mathcal{Z} \to [0, \infty]$ for $Z_{n+1}$ is equivalent to a fuzzy prediction set $\mathcal{E}$ for $Z^{n+1}$: $\mathcal{E}^{z^n}(z)
                            = \mathcal{E}(z^n, z)$.
                    Such a fuzzy prediction set $\mathcal{E}$ is an e-value, and it is valid for $\mathcal{P}$ if and only if it is a valid e-value for $\mathcal{P}$.
                \end{thm}
        
            \subsection{Utility-optimal fuzzy prediction sets}\label{sec:optimal_fuzzy}
                In this section, we study optimal fuzzy prediction sets.
                This generalizes the optimality theory developed in Section \ref{sec:confidence_width_power} from finding an optimal test to finding an optimal e-value.

                In the binary setting, which corresponds to a $\{0, 1/\alpha\}$-valued e-value, only a single objective seems to make sense: maximizing the frequency of hitting $1 / \alpha$.
                Things change when we move beyond the binary setting, since we must express our preferences over the various amounts of evidence that we may obtain.
                We capture this using a utility function $U : [0, \infty] \mapsto [-\infty, \infty]$, maximizing the expected utility of evidence $\E^Q[U(\mathcal{E})]$ under some alternative $Q$.
                Here, we assume the utility function is non-decreasing and concave, expressing that more evidence is better but decreasingly much so.

                We give examples of various utility functions.
                Practical modifications such as clipping or dampening of e-values are discussed in Appendix \ref{sec:clip}.

                \begin{exm}[Log-utility]
                    The e-value literature focuses almost exclusively on log-utility $U = \log$ \citep{grunwald2023safe, larsson2024numeraire, shafer2021testing}. 
                    The log-utility function may be interpreted as valuing evidence linearly in the order of magnitude: $\log(10)$ is half as valuable as $\log(100) = 2\log(10)$.
                    The main motivation for log-utility usually comes from i.i.d. sequential settings: log-utility optimality can be coupled to minimizing expected stopping times through Wald's Identity.
                    Moreover, log-utility is the only utility that is preserved under conditioning; see Proposition 6 in  \citep{koning2026anytime}.
                \end{exm}

                \begin{exm}[Neyman--Pearson-utility]
                    The classical Neyman--Pearson framework is recovered by the `Neyman--Pearson utility' $U_\alpha^{\textnormal{NP}}(x) = x \wedge 1 / \alpha$ \citep{koning2025continuoustestingunifyingtests}.
                    This leads to a linear valuation of evidence up to the threshold $1 / \alpha$, attaching no value to evidence beyond $1 / \alpha$.
                    Since evidence beyond $1 / \alpha$ has no value, this effectively restrains the e-value to $[0, 1/\alpha]$.
                    The familiar Neyman--Pearson randomized testing framework is recovered by rescaling to $[0, 1]$.
                \end{exm}

                \begin{exm}[Capped power-utility]
                    The log-utility and Neyman--Pearson-utility can be united through a two-parameter capped power utility $U_{\alpha, h}(x) = \{(x \wedge 1/\alpha)^{h} - 1\} / h$, $h \leq 1$, $h \neq 0$, and $U_{\alpha, 0}(x) = \log(x \wedge 1/\alpha)$, which appears as the $h \to 0$ limit \citep{koning2025continuoustestingunifyingtests}.
                    Here, the configuration $\alpha = 0$, $h = 0$ recovers $U = \log$ and $\alpha > 0$, $h = 1$ recovers $U^{\textnormal{NP}}$.
                \end{exm}

                While the Neyman--Pearson utility function is (implicitly) the universal standard in classical statistics, we do not believe this truly expresses the valuation of evidence of analysts.
                Moreover, this Neyman--Pearson utility results in prediction sets that are near-binary, with nearly all points assigned either $0$ or $1/\alpha$, as evidenced by the classical likelihood ratio test which emerges from the Neyman--Pearson lemma.
                We believe this is one reason that \citet{geyer2005fuzzy} did not appreciate the potential of fuzzy prediction sets outside of discrete data, as they (implicitly) stuck to such a Neyman--Pearson-style utility function.
                Therefore, we \emph{must} move away from Neyman--Pearson utility to obtain properly fuzzy prediction sets.

                We can translate the maximization of the expected utility of the e-value $\mathcal{E}$ to a statement about a fuzzy prediction set for $Z_{n+1}$: we are maximizing the total utility obtained from excluding each point $z$ at a certain level, as measured by $\mu^{|Z^n}$, in expectation over $Z^n$:
                \begin{align*}
                    \E^{P_*^{Z^n}}\left[\int_{\mathcal{Z}} U(\mathcal{E}^{Z^n}(z))\,d\mu^{|Z^n}(z)\right].
                \end{align*}

                We formalize this in Theorem \ref{thm:most_powerful_e-value}.
                We omit its proof, as it is analogous to that of Theorem \ref{thm:most_powerful}.
                For an illustrative example in the Gaussian setting, see Example~\ref{exm:fuzzy_gaussian} in Appendix~\ref{sec:ex_fuzzy}.
                
                \begin{thm}\label{thm:most_powerful_e-value}
                    Let $\mu^{|Z^n}$ be a $Z^n$-dependent probability kernel, and $U : [0, \infty] \to [-\infty, \infty]$.
                    Then $\mathcal{E}^{Z^n}$ maximizes
                    \begin{align*}
                        \E^{P_*^{Z^n}}\left[\int_{\mathcal{Z}} U(\mathcal{E}^{Z^n}(z))\,d\mu^{|Z^n}(z)\right],
                    \end{align*}
                    if and only if $\mathcal{E}$ is the $U$-expected-utility optimal e-value against the alternative $Q = P_*^{Z^n} \otimes \mu^{|Z^n}$.
                \end{thm}

        \subsection{Merging fuzzy prediction sets}
            Fuzzy prediction sets inherit the desirable merging properties of their underlying e-values.
            
            If we have two arbitrary (possibly dependent) fuzzy prediction sets $\mathcal{E}_1^{Z^n}$ and $\mathcal{E}_2^{Z^n}$ for $Z_{n+1}$, then their average is still a fuzzy prediction set.
            Indeed, their underlying e-values, say $\mathcal{E}_1$ and $\mathcal{E}_2$, still average to a valid e-value: $
                \overline{\mathcal{E}}
                    = (\mathcal{E}_1 + \mathcal{E}_2) / 2.$
            The resulting prediction set $\overline{\mathcal{E}}^{Z^{n}}$ is then a fuzzy prediction set for $Z_{n+1}$.
            In fact, this is equivalent to simply averaging the fuzzy prediction sets $\mathcal{E}_1^{Z^n}$ and $\mathcal{E}_2^{Z^n}$ for $Z_{n+1}$.
            This extends from averaging a finite number of prediction sets to mixtures over collections of prediction sets.
            \citet{wang2025onlyadmissible} recently proved that (weighted) averaging (possibly with the constant e-value $\mathcal{E} \equiv 1$) is the only admissible way to merge arbitrarily dependent e-values.
            \citet{clerico2025simple} provides a simpler proof.

            Such arbitrarily dependent fuzzy prediction sets may be different fuzzy prediction sets produced on the same data.
            For example, this enables us to construct an optimal fuzzy prediction set for many choices of $Q$ and produce a `mixture' fuzzy prediction set that weights our preference over options of $Q$.
            Or it allows us to average over split-conformal fuzzy prediction sets over different choices of the split \citep{vovk2024conformal}.

            As non-fuzzy prediction sets (possibly of different levels) can be viewed as special fuzzy prediction sets, they can also be merged in the same manner, but the merged outcome will generally be a fuzzy prediction set.
            Rounding down such a merged fuzzy prediction set to a non-fuzzy prediction set usually discards a lot of evidence.
            This is equivalent to the procedure recently proposed by \citet{xu2025aggregating} for non-fuzzy prediction sets.

            Another popular method to merge e-values is through multiplication, which is valid under independence.
            While averaging fuzzy prediction sets cannot increase evidence beyond the maximum evidence for any of the individual prediction sets, multiplying them does have this potential.
            At the same time, multiplication may also kill-off evidence if one of the fuzzy prediction sets equals zero at some points.
            Non-fuzzy prediction sets suffer from this issue, as they have zero-values by construction.
            The assumption of independence may be weakened to a kind of conditional mean-independence property of one of the fuzzy prediction sets, given the other fuzzy prediction sets: $\E[\mathcal{E}_i \mid \mathcal{E}_1, \dots, \mathcal{E}_{i-1}, \mathcal{E}_{i+1}, \dots] \leq 1$, $i \geq 1$ \citep{ming2026optimized}.

        \section{Optimal conformal prediction}\label{sec:exchangeable_conformal}
            In this section, we specialize our methodology to conformal prediction, where $\mathcal{P}$ is the class of exchangeable distributions on $\mathcal{Z}^{n+1}$.
            We cover both optimality for classical conformal prediction sets and more general fuzzy conformal prediction sets.

            As we showed in Section \ref{sec:fuzzy_is_different_alpha}, a valid fuzzy conformal prediction set for the model $\mathcal{P}$ corresponds to a valid e-value $\mathcal{E}$ for `hypothesis' $\mathcal{P}$.
            Moreover, in Section \ref{sec:optimal_fuzzy} we show that the fuzzy prediction set is optimal if $\mathcal{E}$ is optimal for some expected-utility target $\E^{Q}[U(\mathcal{E})]$.
            As a result, we may reduce the construction of optimal (fuzzy) conformal prediction sets to constructing optimal e-values for exchangeability.
            This means we may apply recent results on optimal e-values for exchangeability from \citet{koning2023post}.
            
            \subsection{Background: permutations, orbits and exchangeability}
                Let $\Pi(n+1)$ denote the group of permutations on $n+1$ elements.
                This group acts on $\mathcal{Z}^{n+1}$ by permuting the elements of each tuple $(z_1, \dots, z_{n+1}) \in \mathcal{Z}^{n+1}$.
                The group action partitions $\mathcal{Z}^{n+1}$ into a set $\mathcal{O}^{n+1}$ of disjoint \emph{orbits} $O \in \mathcal{O}^{n+1}$.
                For a given $z^{n+1} \in \mathcal{Z}^{n+1}$, its orbit is the set of all its permutations:
                \begin{align*}
                    O(z^{n+1})
                        = \{z \in \mathcal{Z}^{n+1} : z = \pi z^{n+1}, \textnormal{ for some } \pi \in \Pi(n+1)\}.
                \end{align*}
                On each orbit, we designate a single tuple as its \emph{orbit representative}, and we define the map $[\cdot]$ as the map that carries any point on the orbit to this representative.

                We say that a random variable $Z^{n+1}$ on $\mathcal{Z}^{n+1}$ is exchangeable if
                \begin{align*}
                    Z^{n+1} \overset{d}{=} \pi Z^{n+1}, \textnormal{ for every } \pi \in \Pi(n+1).
                \end{align*}
                Note that this is actually a statement about its distribution $P^{Z^{n+1}}$, which we say is exchangeable if it is the law of an exchangeable random variable.
                With $\mathcal{P}^{\textnormal{exch.}}$, we denote the collection of exchangeable distributions on $\mathcal{Z}^{n+1}$.
                Throughout, we use the following well-known result.
                \begin{lem}\label{lem:representation_invariance}
                    $P$ is exchangeable if and only if $P^{|O} = \textnormal{Unif}(O)$, $P^{O}$-a.s.
                \end{lem}

            \subsection{Optimal conformal}
                The key idea to derive optimal e-values for exchangeability is to reduce to a problem on each orbit.
                Indeed, Theorem \ref{thm:representation_optimality}, shows that if $\mathcal{E}$ is `locally' valid and optimal on each orbit, then it is also `globally' valid and optimal.
                The result follows from Theorem 5 in \citet{koning2023post},

                \begin{thm}\label{thm:representation_optimality}
                    If $\mathcal{E}$ is valid for $\textnormal{Unif}(O)$ and optimal for $\E^{Q^{|O}}[U(\cdot)]$, for every $O$, then $\mathcal{E}$ is valid for $\mathcal{P}^{exch.}$ and optimal against $\E^{Q}[U(\mathcal{E})]$.
                    Moreover, $\mathcal{E}$ is optimal uniformly in mixtures over $Q^{| O}$.
                \end{thm}

                To establish orbitwise optimality, we use the ``Neyman--Pearson lemma for e-values'' recently derived in \citet{koning2025continuoustestingunifyingtests}.
                The classical Neyman--Pearson lemma corresponds to the choice $U(x) = x \wedge 1/\alpha$.
                The version we present in Theorem \ref{thm:generalized-NP-lemma} specializes to the structure at hand: $O$ is finite and $Q^{|O}$ is absolutely continuous with respect to $\textnormal{Unif}(O)$, since $\textnormal{Unif}(O)$ has full support.
                Corollary \ref{cor:differentiable} presents a more easily interpretable version under some regularity conditions on $U$ and $Q^{|O}$.

                \begin{thm}\label{thm:generalized-NP-lemma}
                    Let $U : [0, \infty] \to [-\infty, \infty]$ be concave, non-decreasing and upper-semicontinuous.
                    Then, the optimization problem
                    \begin{align*}
                        \sup_{\mathcal{E}} \E^{Q^{|O}}[U(\mathcal{E})], \textnormal{ s.t. } \mathcal{E} : \mathcal{X} \to [0, \infty],\ \E^{\textnormal{Unif}(O)}[\mathcal{E}] \leq 1,
                    \end{align*}
                    admits an optimal solution.
                    Moreover, if $\mathcal{E}^*$ is an optimizer, then there exists a normalization constant $\lambda_O \geq 0$ such that
                    \begin{align*}
                        \lambda_O \bigg/ \frac{dQ^{|O}}{d\textnormal{Unif}(O)} &\in \partial U(\mathcal{E}^*),  
                    \end{align*}
                    on $\{Q^{|O} > 0\}$ and $\mathcal{E}^* = 0$ on $\{Q^{|O} = 0\}$, where $\partial U$ is the superdifferential of $U$.
                \end{thm}

                \begin{cor}\label{cor:differentiable}
                    If $Q^{|O}$ has full support on $O$ and $U$ is differentiable with strictly decreasing derivative $U'$, then an optimizer is
                    \begin{align*}
                        \mathcal{E}^* = (U')^{-1}\left(\lambda_O \bigg/ \frac{dQ^{|O}}{d\textnormal{Unif}(O)}\right),
                    \end{align*}
                    where $(U')^{-1}$ denotes the functional inverse of $U'$.
                \end{cor}

                A consequence of Theorem \ref{thm:generalized-NP-lemma} is that the restriction of $\mathcal{E}_{|O}^*$ of $\mathcal{E}^*$ to the orbit $O$ only depends on the data through the orbit-conditional likelihood ratio $\textnormal{LR}^{|O} := dQ^{|O} / d\textnormal{Unif}(O)$.
                In Proposition \ref{prp:conditional_LR}, we use the structure of $Q = P^{Z^n} \otimes \mu^{|Z^n}$ to show that this conditional likelihood ratio only depends on $Z^{n+1}$ through $Z_{n+1}$.
                
                \begin{prp}\label{prp:conditional_LR}
                    If $Z^n$ is exchangeable under $Q$, then 
                    \begin{align*}
                        \textnormal{LR}^{|O}(z^{n+1})
                            = \frac{dQ^{|O}}{d\textnormal{Unif}(O)}(z^{n+1}) 
                            = \frac{Q^{Z_{n+1} | O}(z_{n+1})}{1/(n+1)}
                            =: \textnormal{LR}^{Z_{n+1}|O}(z_{n+1}).
                    \end{align*}
                \end{prp}

                A consequence of Proposition \ref{prp:conditional_LR} is that the restriction $\mathcal{E}_{|O}^*$ does not depend on $Z^n$.
                Indeed, we are implicitly considering the following hypotheses on each orbit:
                \begin{align*}
                    H_0^O : Z_{n+1} \sim \textnormal{Unif}[Z^{n+1}], \quad
                    H_1^O : Z_{n+1} \sim Q^{Z_{n+1}|O}.
                \end{align*}
                By extension, the global e-value $\mathcal{E}^*$ only depends on $Z^n$ through the orbit $O(Z^{n+1})$.
                
                \begin{rmk}[Conformal with orbit-level ranks]\label{rmk:conditional_ranks}
                    As optimal e-values only depend on the position of $Z_{n+1}$ in $[Z^{n+1}]$, we may equivalently express all the results in terms of the rank of $Z_{n+1}$ among $[Z^{n+1}]$.
                    To define a rank for arbitrary (non-numerical) data, we may define it relative to the orbit representative $[Z^{n+1}]$.
                    In particular, we define $\textnormal{Rank}(Z^{n+1})$ as the index at which each element in the tuple $Z^{n+1}$ needs to be placed to obtain the representative $[Z^{n+1}]$ of $O$.
                    This specializes to the conventional rank for data that has a natural order, if $[Z^{n+1}]$ is sorted in ascending order.

                    Let $R_{n+1}^O$ denote last element of the tuple $\textnormal{Rank}(Z^{n+1})$, which corresponds to the rank of $Z_{n+1}$ among $[Z^{n+1}]$.
                    We may then equivalently express the reduced orbit-level hypotheses as
                    \begin{align*}
                        H_0^O : R_{n+1}^O \sim \textnormal{Unif}\{1, \dots, n+1\}, \quad 
                        H_1^O : R_{n+1}^O \sim Q_R^O,
                    \end{align*}
                    where $Q_R^O$ is some arbitrary orbit-dependent distribution on $\{1, \dots, n+1\}$.

                    In the literature on conformal prediction it is popular to express everything in terms of unconditional ranks (sometimes confusingly named p-values).
                    Using such unconditional ranks exactly discards the information about the orbit.
                \end{rmk}
        
            \subsection{Illustration: i.i.d. and split-conformal}\label{sec:Q_to_LR}
                As the conditional likelihood ratio is central in finding the optimal e-value, we first illustrate what such a conditional likelihood ratio may look like in a popular setting, before we proceed with several examples of utility functions.

                In Proposition \ref{prp:conditional_LR_indep}, we consider optimality against an i.i.d. alternative of the form $Q = P^{Z^n} \times \mu$ where $P^{Z^{n}} = (P^{Z_1})^{n}$.
                For example, such an alternative may arise by plugging-in an i.i.d. estimator $\widehat{P}^{Z^n} = (\widehat{P}^{Z_1})^{n}$ for $P_*^{Z^{n}}$, based on a separate sample akin to the popular `split-conformal prediction' approach.

                \begin{prp}\label{prp:conditional_LR_indep}
                    Suppose that $Q$ is so that $Z^{n+1}$ is independent and $Z^n$ is i.i.d., with $Q^{Z_{n+1}}$ and $Q^{Z_1}$ mutually absolutely continuous.
                    Then,
                    \begin{align*}
                        &\textnormal{LR}^{Z_{n+1}|O(z^{n+1})}(z_{n+1}) 
                            = \frac{dQ^{Z_{n+1}}}{dQ^{Z_{1}}}(z_{n+1}) \Bigg/ \left(\frac{1}{n+1} \sum_{i = 1}^{n+1} \frac{dQ^{Z_{n+1}}}{dQ^{Z_{1}}}(z_{i})\right).
                    \end{align*}
                \end{prp} 

                Proposition \ref{prp:conditional_LR_indep} may be interpreted as stating that
                \begin{align*}
                    \textnormal{LR}^{Z_{n+1}|O(z^{n+1})}(z_{n+1})
                        \propto \frac{d\mu}{d\widehat{P}^{Z_{1}}}(z_{n+1}),
                \end{align*}
                where the proportionality is up to some orbit-dependent constant.
                Viewing $\mu$ as some reference measure, an interpretation of this result is that the optimal e-value is decreasing in the marginal $\widehat{P}^{Z_1}$-density at $z_{n+1}$.
                We believe this makes intuitive sense, as values of $z_{n+1}$ that have low density may be interpreted as being `outlying' or `non-conforming'.

            \subsection{Log-utility and power utility}
                In the e-value literature, the near-universal choice of utility is the log-utility $U : x \mapsto \log(x)$, for which the optimal e-value is the orbit-conditional likelihood ratio itself:
                \begin{align*}
                    \mathcal{E}^{\log}(z^{n+1})
                        = \textnormal{LR}^{Z_{n+1}|O(z^{n+1})}(z_{n+1}).
                \end{align*}
                This generalizes to the power-utility $U : x \mapsto (x^h - 1)/h$, $h < 1$, $h \neq 0$, with optimizer
                \begin{align*}
                    \mathcal{E}^h(z^{n+1})
                        = \frac{(\textnormal{LR}^{Z_{n+1}|O(z^{n+1})}(z_{n+1}))^{\frac{1}{1-h}}}{\frac{1}{n+1}\sum_{j = 1}^{n+1} (\textnormal{LR}^{Z_{n+1}|O(z^{n+1})}(z_{j}))^{\frac{1}{1-h}}}.
                \end{align*}
                The log-utility appears as the $h \to 0$ limit.

                \begin{rmk}
                    In the setting of Section \ref{sec:Q_to_LR}, this amounts to choosing $T(z_{n+1}) = \left(\frac{d\mu}{d\widehat{P}^{Z_{1}}}(z_{n+1})\right)^{1 / (1 - h)}$ in \eqref{eq:T-based_e-value}.
                \end{rmk}

            \subsection{Optimal classical conformal}\label{sec:classical_conformal}
                 In Theorem \ref{thm:NP_style}, we show how our optimality framework nests classical conformal prediction as a special case if we choose the Neyman--Pearson utility function $U_\alpha^{\textnormal{NP}} : x \mapsto x \wedge 1/\alpha$, $\alpha \in (0, 1)$.
                 This nests classical conformal prediction for any non-conformity score that is a monotone function of $\textnormal{LR}^{Z_{n+1}|O(z^{n+1})} = (n + 1) \times Q^{Z_{n+1} | O(z^{n+1})}$.
            
                \begin{thm}\label{thm:NP_style}
                    Write $O'$ for $O(z^{n+1})$.
                    A $U_\alpha^{\textnormal{NP}}$-optimal e-value is
                    \begin{align*}
                        \mathcal{E}^{\textnormal{NP}}(z^{n+1})
                            =
                            \begin{cases}
                                1/\alpha, &\textnormal{ if } \textnormal{LR}^{Z_{n+1}|O'}(z_{n+1}) > c_\alpha^{O'}, \\
                                k^{O'}, \hspace{-.4cm}&\textnormal{ if } \textnormal{LR}^{Z_{n+1}|O'}(z_{n+1}) = c_\alpha^{O'}, \\
                                0, &\textnormal{ if } \textnormal{LR}^{Z_{n+1}|O'}(z_{n+1}) < c_\alpha^{O'},
                            \end{cases}
                    \end{align*}
                    where $c_\alpha^{O'}$ is the $\alpha$ upper-quantile of $\textnormal{LR}^{|O'}(z^{n+1})$ under $\textnormal{Unif}[Z^{n+1}]$, and $k^{O'}$ is some orbit-dependent constant that ensures $\mathcal{E}^{\textnormal{NP}}(Z^{n+1})$ is an exact e-value.
                    This e-value is optimal uniformly in any mixture over orbits, and in monotone transformations of the likelihood ratio.
                \end{thm}

                To make this more interpretable, we provide Corollary \ref{cor:classical_conformal}, which combines Theorem \ref{thm:NP_style} with our result for the i.i.d. setting in Proposition \ref{prp:conditional_LR_indep}.
                Here, we consider an i.i.d. estimator $\widehat{P}^{Z^n} = (\widehat{P}^{Z_1})^n$ for $P_*^{Z^n}$, and minimize the expected measure of the prediction set
                \begin{align*}
                    \E^{\widehat{P}^{Z^n}}[\mu(C_\alpha^{Z^n})].
                \end{align*}
                Using $f = d\widehat{P}^{Z_{1}} / d\mu$ to denote the density, Corollary \ref{cor:classical_conformal} shows that classical conformal prediction is optimal in this sense when the conformity score $s$ is chosen as a monotone function of the density $f$.
                Conversely, given a conformity score $s$ and measure $\mu$, classical conformal prediction is optimal uniformly in the class of true distributions for which $s$ is a monotone function of the corresponding density $f$.

                \begin{cor}\label{cor:classical_conformal}
                    Suppose $\widehat{P}^{Z^{n}} = (\widehat{P}^{Z_1})^n$, and let $s : \mathcal{Z} \to \mathbb{R}_+$ denote a conformity score.
                    Then, classical conformal prediction minimizes the expected prediction set size for any $\mu$ and $\widehat{P}^{Z_{1}}$ that satisfy $s = m\left(\frac{d\widehat{P}^{Z_{1}}}{d\mu}\right)$, where $m : [0, \infty] \to [-\infty, \infty]$ is some strictly monotone function.
                \end{cor}

        \section{Certifying subsequent decisions}\label{sec:certified_decisions}
            In this section, we motivate (fuzzy) prediction sets through the loss bounds they provide to decision-makers who face uncertainty about the outcome of $Z_{n+1}$.
            
            We consider a decision-maker who must select a decision $d$ from a decision space $\mathcal{D}$.
            To express their preferences, we consider a loss function $L_{z} : \mathcal{D} \to \mathbb{R}$.
            Here, they would ideally minimize the `oracle loss' $L_{Z_{n+1}}$ associated with true value of $Z_{n+1}$.
            Unfortunately, the outcome of $Z_{n+1}$ is not yet available at the time of decision making.
            Instead, the decision-maker relies on a (fuzzy) prediction set to inform themselves about $Z_{n+1}$.

            \subsection{Loss bounds from prediction sets}\label{sec:as-if}
                Suppose the decision-maker receives a classical (non-fuzzy) prediction set $C_\alpha^{Z^n}$ for $Z_{n+1}$ that is valid at level $\alpha$ under some model $\mathcal{P}$.
                We may then convert this prediction set into a loss bound $\sup_{z \in \overline{C}_\alpha^{Z^n}} L_z(\delta)$, for any possibly $Z^n$-dependent decision rule $\delta$.

                \begin{prp}\label{prp:minimax}
                    For every $Z^n$-dependent rule $\delta$,
                    \begin{align}\label{ineq:loss_bound}
                        P\left(L_{Z_{n+1}}(\delta) > \sup_{z \in \overline{C}_\alpha^{Z^n}} L_z(\delta)\right) \leq \alpha,\quad \textnormal{for every } P \in \mathcal{P}.
                    \end{align}
                \end{prp}

                The tightest loss bound is obtained by the minimax decision rule
                \begin{align}\label{eq:minimax}
                    \widehat{\delta} \in \argmin_{d \in \mathcal{D}} \sup_{z \in C_\alpha^{Z^n}} L(d, z),
                \end{align}
                assuming such a minimizer exists.
                Loss bounds of this type were recently studied by \citet{andrews2025certified} for classical confidence sets and \citet{kiyani2025decision} for conformal prediction.
                
            \subsection{Loss bounds from fuzzy prediction sets}
                We now show how \emph{fuzzy} prediction sets yield richer loss bounds.
                We consider two types of loss bounds: post-hoc loss bounds and a weighted loss bounds.
                        
                \subsubsection{Post-hoc loss bounds}
                    A downside of decisions based on non-fuzzy prediction sets is that the confidence level $\alpha$ must be prespecified, but may not match the level of certainty desired by the decision-maker.
                    To resolve this, we combine the loss bound \eqref{ineq:loss_bound} with post-hoc validity \eqref{dfn:post-hoc_valid}.
                    
                    To present the resulting bound, recall that a valid fuzzy prediction set $\mathcal{E}^{Z^n}$ is equivalent to a post-hoc valid collection $(\overline{C}_\alpha^{Z^n})_{\alpha > 0}$ of prediction sets $\overline{C}_\alpha^{Z^n} = \left\{z \in \mathcal{Z} : \mathcal{E}^{Z^n}(z) < 1/\alpha\right\}$.
                    For a given $Z^n$-dependent decision $\delta$, this yields a range $(\overline{L}_\alpha(\delta))_{\alpha > 0}$ of loss bounds $\overline{L}_\alpha(\delta) = \sup_{z \in \overline{C}_\alpha^{Z^n}} L_z(\delta)$ over different confidence levels.

                    To the best of our knowledge, we are the first to consider post-hoc loss bounds of this type.
                    The proof follows from the same implication as used in the proof of Proposition \ref{prp:minimax}.

                    \begin{prp}
                        For every $Z^n$-dependent decision rule $\delta$ and data-dependent level $\widetilde{\alpha}$,              
                        \begin{align}\label{ineq:post-hoc_loss_bound}
                            \E_{\widetilde{\alpha}}^P 
                        \left[
                            \frac
                                {P\left(L_{Z_{n+1}}(\delta) \geq \sup_{z \in \overline{C}_{\widetilde{\alpha}}^{Z^n}} L_z(\delta) \mid \widetilde{\alpha}\right)}
                                {\widetilde{\alpha}}\right] 
                        \leq 1, \textnormal{ for every } P \in \mathcal{P}.
                        \end{align}
                    \end{prp}

                    \begin{rmk}[Post-hoc minimax decisions]
                        A concrete way to use the loss bound \eqref{ineq:post-hoc_loss_bound}, is to consider the minimax decision $\widehat{\delta}_\alpha \in \argmin_{d \in \mathcal{D}} \sup_{z \in \overline{C}_\alpha^{Z^n}} L_z(\delta)$ for each confidence level $\alpha > 0$.
                        These produce a range of loss-confidence pairs $(\widehat{L}_\alpha)_{\alpha > 0}$ which the decision-maker can browse to select the desired decision, $\widehat{L}_\alpha := \overline{L}_\alpha(\widehat{\delta}_\alpha)$.
                    \end{rmk}
            
            \subsubsection{Evidence-weighted loss bounds}\label{sec:weighted_loss}
                A downside of loss bounds of the form $\sup_{z \in C_\alpha^{Z^n}} L_z(\delta)$ is that they hinge on a hard set inclusion $z \in C_\alpha^{Z^n}$.
                An alternative approach is to consider weighted loss bound $\sup_{z \in \mathcal{Z}} L_z(\delta) / \mathcal{E}^{Z^n}(z)$, assuming $L_z, \mathcal{E}^{Z^n} > 0$.

                \begin{prp} \label{prop: data-dep}
                    For every $Z^n$-dependent decision rule $\delta$
                    \begin{align}\label{ineq:E-certificate}
                        \E^P\left[\frac{L_{Z_{n+1}}(\delta)}{\sup_{z \in \mathcal{Z}} L_{z}(\delta) / \mathcal{E}^{Z^n}(z)} \right] \leq 1, \textnormal{ for every } P \in \mathcal{P}.
                    \end{align}
                \end{prp}

                \begin{rmk}[Evidence-weighted minimax]
                    The corresponding minimax decision $
                    \delta \in \argmin_{d \in \mathcal{D}} \sup_{z \in \mathcal{Z}} L_{z}(d) / \mathcal{E}^{Z^n}(z)$ downweights values of $z$ against which we have much evidence, as such values are unlikely to become the realization of $Z_{n+1}$.
                    This emphasizes plausible realizations of $Z_{n+1}$.
                \end{rmk}

                Evidence-weighted loss bounds were introduced by \citet{grunwald2023posterior} on a parameter space, and the admissibility of weighted minimax decisions was studied by \citet{andrews2025certified}.
                Our contribution is to develop a predictive analogue.
                
                In Appendix \ref{app:connection}, we show that \eqref{ineq:E-certificate} can be viewed as a generalization of \eqref{ineq:loss_bound}.
       
        \section{Application: image recognition}\label{sec:application}
            We now showcase our methodology in a classical Machine Learning problem: character recognition.
            We use the same data and model as considered by \citet{gauthier2025values}.
            
            In particular, we consider the Federated Extended MNIST (FEMNIST) dataset \citep{caldas2018leaf}.
            In this dataset, $Y_i$ is one of 62 possible characters (a-z, A-Z, 0-9) and $X_i$ is a $28 \times 28$ pixel image of a handwritten version of this character.
            We split the data into a training set (80\%), calibration set (15\%) and test set (5\%), where the percentages are approximate because we put characters by the same writer into the same subset.

            We use the training data to `train' the same LeNet-inspired \citep{lecun1998mnist} neural network as used by \citet{gauthier2025values}, which is an estimator of the kernel $P_*^{Y_{n+1}|X_{n+1}}$.
            We plug this estimator into the likelihood ratio with covariates, as described in Appendix \ref{sec:conformal_covariates}, and subsequently use this to construct several utility-optimal e-values, as described in Section \ref{sec:exchangeable_conformal}.
            The utilities we consider are all special cases of the capped power-utility $U : x \mapsto ((x \wedge 1/\alpha)^h - 1) / h$ framework introduced by \citet{koning2025continuoustestingunifyingtests}, which interpolates between log-optimal e-value ($\alpha = 0$ and $h \to 0$) and Neyman--Pearson-optimal e-value ($\alpha > 0$ and $h = 1$).
            We compute the fuzzy prediction set by taking the calibration set to be $Z^n$, and we select the image of a single observation from the test set to be $X_{n+1}$, so that our fuzzy prediction sets are on the true label $Y_{n+1}$.

            In Figure \ref{fig:fuzzy_FEMNIST}, we plot the resulting fuzzy prediction sets for the label $Y_{n+1}$ of the image $X_{n+1}$ given in Figure \ref{fig:character}.
            Here, each bar represents the amount of evidence (e-value) against every character.
            The plots are sorted from least-to-most evidence, on which they all agree (up to ties) as these utility-optimal e-values are all non-decreasing functions of the same likelihood ratio \citep{koning2025continuoustestingunifyingtests}.
            Here, we see that the parameters $\alpha$ and $h$ are jointly able to change the shape of the fuzzy prediction set quite dramatically.

            The approximate staircase-shape of the capped power-utility plot (bottom-left) makes it a good starting point to explain how to interpret these plots.
            Ignoring the slightly elevated evidence at characters $G$, $b$ and $a$, this fuzzy prediction set can be interpreted as \emph{simultaneously} reporting multiple prediction sets for different data-dependent level certificates $\widetilde{\alpha}$:
            \begin{align*}
                C_{\widetilde{\alpha}}
                    =
                    \begin{cases}
                        \{d\}, & \textnormal{ for } 0.01 \leq \widetilde{\alpha}, \\
                        \{d, 0, G, b\}, & \textnormal{ for } 0.00125 \leq \widetilde{\alpha} < 0.01, \\
                        \{d, 0, G, b, J, a\}, & \textnormal{ for } 0.001 \leq \widetilde{\alpha} < 0.00125, \\
                        \textnormal{all characters}, & \textnormal{ for } \widetilde{\alpha} < 0.001.
                    \end{cases}
            \end{align*}
            We may contrast this to the Neyman--Pearson plot (bottom-right), for which this becomes:
            \begin{align*}
                C_{\widetilde{\alpha}}
                    =
                    \begin{cases}
                        \{d, 0, G, b\}, & \textnormal{ for } 0.001 \leq \widetilde{\alpha}, \\
                        \textnormal{all characters}, & \textnormal{ for } \widetilde{\alpha} < 0.001.
                    \end{cases}
            \end{align*}
            This shows that a fuzzy prediction set yields a much more refined and non-binary expression of the available evidence.
            Comparing the two, a benefit of this fuzzy prediction set is that it reports a prediction set $\{d\}$ at a certificate of level $0.01$, which is unavailable for the Neyman--Pearson prediction set.
            The cost is that $\{d, 0, G, b\}$ only comes with a certificate of $0.00125$, instead of the $0.001$ offered by the Neyman--Pearson variant.
            
            The capped log-utility (top-right) and especially the uncapped log-utility (top-left) show a much wider spectrum of evidence.
            The uncapped log-utility displays strong evidence against characters such as $9$ and $R$, which is not visible in the other fuzzy prediction sets that are capped at $1000$.
            This may be of interest if the subsequent task is to eliminate certain characters from contention, instead of determining the correct character.
            Indeed, the log-utility maximizer reports at least 13 characters for post-hoc significance levels $\widetilde{\alpha} < 0.1$, which does not seem helpful if we desire to identify the `correct' character.
            This suggests that log-utility, which is often pushed as `the right utility function' in the e-value literature, is not necessarily ideal for constructing fuzzy prediction sets.
            At the same time, such behavior may be attractive in different applications, such as a preliminary medical diagnosis to help rule out several implausible diseases.

            An important open question is what kind of utility functions give rise to desirable fuzzy prediction sets for certain subsequent decision tasks.
            For now, the parameters $\alpha$ and $h$ seem like two useful instruments to shape our utility function, and thereby the resulting prediction sets: $\alpha$ caps the maximum evidence that we are interested in, and $h \in [-\infty, 1]$ influences how `risky' or `aggressive' the e-value is.

            \begin{figure*}[t]
                \centering
                \includegraphics[width=\textwidth]{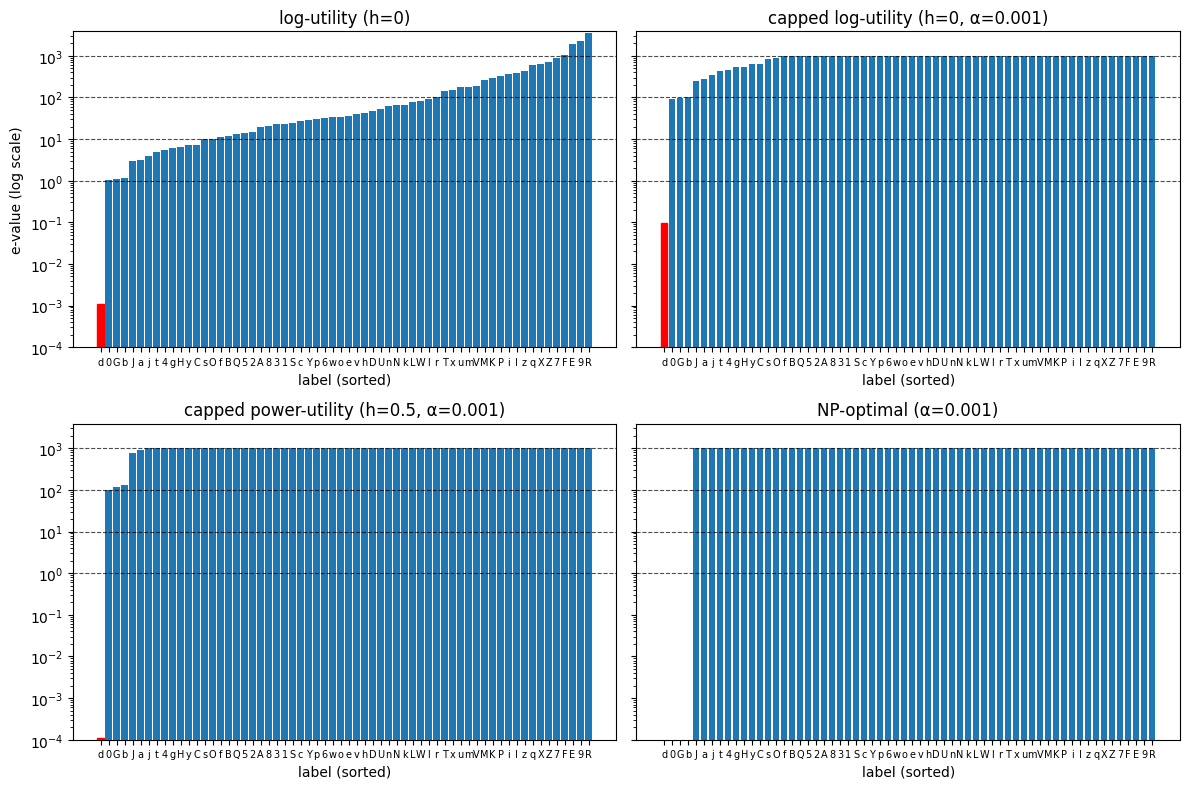}
                \caption{Expected-power-utility-optimal $[0, 1/\alpha]$-valued fuzzy prediction sets over possible labels, for $h = 0$ (log-utility), $h = 1/2$ (power-utility) and $h = 1$ (Neyman--Pearson), for values $\alpha = 0$ (top-left) and $\alpha = 0.001$ (others). The true character label is marked in red (d).}
                \label{fig:fuzzy_FEMNIST}
            \end{figure*}

            \begin{figure}
                \centering
                \includegraphics[width=4cm]{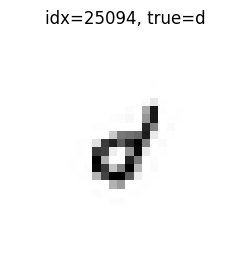}
                \caption{Handwritten character used in the application.}
                \label{fig:character}
            \end{figure}

        \section{Data availability}
            Code to replicate the application may be found at \url{https://github.com/nickwkoning/fuzzy_prediction_sets}.
            The data used in this study are derived from publicly available resources. 
            Specifically, we use the FEMNIST dataset from the LEAF benchmark introduced by \cite{caldas2018leaf}. 
            The dataset is available in the public domain via its official repository: \url{https://github.com/TalwalkarLab/leaf}.

        \section{Acknowledgements}
            We thank Etienne Gauthier, Peter Gr\"unwald, Guneet Dhillon, Yash Nair and Junu Lee for their feedback and discussions.
            We acknowledge the use of ChatGPT-5 for AI-assisted proofreading.
            Nick Koning is supported by a starter grant from the Dutch government.

        {
        \spacingset{1}
        \small
        \bibliographystyle{plainnat}
        \bibliography{Bibliography-MM-MC}
        }
        \appendix
        \clearpage

        % Roman numbering for appendix sections
        \pagenumbering{roman}

        % Supplementary title (no authors)
        \begin{center}
            {\Large\bf Supplementary Materials for\\[0.3em]
            Fuzzy Prediction Sets: Conformal Prediction with E-values}
        \end{center}
        \vspace{.0em}

        \section{Omitted proofs}\label{sec:proofs}
            \subsection{Proof of Theorem \ref{thm:most_powerful}}
                \begin{proof}
                    This result follows immediately from the fact that we can represent a prediction set $C_\alpha$ as a test $C_\alpha : \mathcal{Z}^{n+1} \to \{0, 1\}$:
                        \begin{align*}
                        &\argmin_{C_\alpha}\E^{P_*^{Z^n}}[\mu^{|Z^n}(C_\alpha^{Z^{n}})] \\
                            &= \argmin_{C_\alpha} \int_{\mathcal{Z}^{n}}\int_{\mathcal{Z}} \mathbb{I}\{Z_{n+1} \in C_\alpha^{Z^n}\}\,d\mu^{|Z^n}\,dP_*^{Z^n}  \\
                            &= \argmin_{C_\alpha} \int_{\mathcal{Z}^{n+1}} (1 - C_\alpha)\,d(P_*^{Z^n} \otimes \mu^{|Z^n}) \\
                            &= \argmax_{C_\alpha} \int_{\mathcal{Z}^{n+1}} C_\alpha \,d(P_*^{Z^n} \otimes \mu^{|Z^n}) \\
                            &= \argmax_{C_\alpha} \E^{Q}[C_\alpha],
                    \end{align*}
                    where the third equality follows from the fact that $\mu^{|Z^n}$ is a probability kernel.
                \end{proof}
                
            \subsection{Proof of Theorem \ref{thm:post-hoc}}\label{proof:post-hoc}
                 \begin{proof}
                    For each $\overline{C}_\alpha$, \eqref{eq:sublevel_sets} reveals that the smallest data-dependent level $\widetilde{\alpha}$ choice for which $Z_{n+1}$ is excluded is $1/\mathcal{E}^{Z^n}(Z_{n+1})$.
                    This implies, 
                    \begin{align*}
                        \frac{\mathbb{I}\{Z_{n+1} \not\in \overline{C}_{\widetilde{\alpha}}^{Z^n}\}}{\widetilde{\alpha}}
                            \leq \frac{\mathbb{I}\{Z_{n+1} \not\in \overline{C}_{1/\mathcal{E}^{Z^n}(Z_{n+1})}\}}{1/\mathcal{E}^{Z^n}(Z_{n+1})}
                            = \mathcal{E}^{Z^n}(Z_{n+1}),
                    \end{align*}
                    for every $\widetilde{\alpha}$.
                    As a consequence,
                    \begin{align*}
                         \E_{\widetilde{\alpha}}^P&\left[\frac{P(Z_{n+1} \not\in \overline{C}_{\widetilde{\alpha}}^{Z^n} \mid \widetilde{\alpha})}{\widetilde{\alpha}}\right] \\
                            &= \E_{\widetilde{\alpha}}^P\left[\E^P\left[\frac{\mathbb{I}\{Z_{n+1} \not\in \overline{C}_{\widetilde{\alpha}}^{Z^n} \}}{\widetilde{\alpha}}\middle| \widetilde{\alpha}\right]\right] \\
                            &= \E^P\left[\frac{\mathbb{I}\{Z_{n+1} \not\in \overline{C}_{\widetilde{\alpha}}^{Z^n} \}}{\widetilde{\alpha}}\right]
                            \leq \E^P\left[\mathcal{E}^{Z^n}(Z_{n+1})\right] \leq 1.
                    \end{align*}
                \end{proof}
                
            \subsection{Proof of Proposition 
            \ref{prp:conditional_LR}}\label{proof:conditional_alternative}
                To prove Proposition \ref{prp:conditional_LR}, we first prove the following lemma.
                
                \begin{lem}
                    If $Z^n$ is exchangeable under $Q$ then
                    \begin{align}
                        Q^{Z^n | O, Z_{n+1}} = \textnormal{UnifWR}([Z^{n+1}] \setminus \{Z_{n+1}\}),
                    \end{align}
                    where with $\textnormal{UnifWR}([Z^{n+1}] \setminus \{Z_{n+1}\})$ we mean uniformly sampled without replacement from the tuple $[Z^{n+1}]$ with $Z_{n+1}$ removed.
                \end{lem}
                \begin{proof}
                    By assumption, $Z^{n}$ is exchangeable.
                    Exchangeability of $Z^n$ means that the distribution of $Z^{n+1}$ is invariant under permutations in $\Pi(n)$, which is equivalent to uniformity on each $\Pi(n)$-orbit of $\mathcal{Z}^{n+1}$, as $\Pi(n)$ is a compact group.
                    Each $\Pi(n)$-orbit of $\mathcal{Z}^{n+1}$ is simply the subset of a $\Pi(n+1)$-orbit $O \in \mathcal{O}^{n+1}$ in which the $(n+1)$th observation is fixed.
                    In fact, each $\Pi(n+1)$-orbit may be partitioned into $\Pi(n)$-orbits as subsets, each differing by the $(n+1)$th element.
                    Hence, conditionally on $O \in \mathcal{O}^{n+1}$, each conditional distribution $Q^{|O}$ may be viewed as first sampling this $(n+1)$th element using \emph{some} distributions (and with it the $\Pi(n)$-orbit), and subsequently sampling uniformly from the $\Pi(n)$-orbit that contains this element.
                \end{proof}

                \begin{proof}[Proof of Proposition \ref{prp:conditional_LR}]
                    We have
                    \begin{align*}
                        &\textnormal{LR}^{|O}(z^{n+1})
                            = \frac{dQ^{|O}}{d\textnormal{Unif}(O)}(z^{n+1}) \\
                            &= \frac{d(Q^{Z_{n+1}|O} \otimes \textnormal{UnifWR}([Z^{n+1}] \setminus \{Z_{n+1}\}) )}{d(\textnormal{Unif}[Z^{n+1}] \otimes \textnormal{UnifWR}([Z^{n+1}] \setminus \{Z_{n+1}\}))}(z^{n+1}) \\
                            &= \frac{dQ^{Z_{n+1}|O}}{d\textnormal{Unif}[Z^{n+1}]}(z_{n+1}) 
                            = (n + 1) \times Q^{Z_{n+1} | O}(z_{n+1}) \\
                            &=: \textnormal{LR}^{Z_{n+1}|O}(z_{n+1}).
                    \end{align*}
                \end{proof}
                
            \subsection{Proof of Proposition \ref{prp:conditional_LR_indep}}\label{proof:iid}
                \begin{proof}
                    By assumption, $Q = Q_1 \times Q_1 \times \cdots \times Q_1 \times Q_{n+1}$.
                    Define $Q^{(i)}$ be this product measure, but with $Q_{n+1}$ in the $i$th position, and $Q_1$ in the remaining positions.
                    By the structure of $Q$,
                    \begin{align*}
                        \overline{Q}
                            = \frac{1}{|\Pi(n+1)|}\sum_{\pi \in \Pi(n+1)} \pi Q
                            = \frac{1}{n+1} \sum_{i=1}^{n+1} Q^{(i)}.
                    \end{align*}

                    By Proposition 3 in \citet{koning2023post}, the conditional likelihood ratio is the restriction of
                    \begin{align*}
                        \frac{dQ}{d\overline{Q}}
                    \end{align*}
                    to $O$.
                    By definition of $\overline{Q}$, we have, for every event $A$,
                    \begin{align*}
                        \overline{Q}(A)
                            &= \frac{1}{n + 1} \sum_{i=1}^{n+1} Q^{(i)}(A) \\
                            &= \frac{1}{n + 1} \sum_{i=1}^{n+1} \int_{A} \frac{dQ^{(i)}}{dQ}\,dQ \\
                            &= \int_{A} \left(\frac{1}{n + 1} \sum_{i=1}^{n+1}  \frac{dQ^{(i)}}{dQ}\right)\,dQ,
                    \end{align*}
                    so that
                    \begin{align*}
                        \frac{d\overline{Q}}{dQ}
                            = \frac{1}{n + 1} \sum_{i=1}^{n+1}  \frac{dQ^{(i)}}{dQ}.
                    \end{align*}
                    Taking the reciprocal gives
                    \begin{align}\label{eq:average_RN}
                        \frac{dQ}{d\overline{Q}}
                            = \frac{1}{\frac{1}{n + 1} \sum_{i=1}^{n+1}  \frac{dQ^{(i)}}{dQ}}.
                    \end{align}
                    Next, by definition of $Q$ and $Q^{(i)}$, we have
                    \begin{align*}
                        \frac{dQ^{(i)}}{dQ}(z^{n+1})
                            &= \frac{\tfrac{dQ_{n+1}}{dQ_1}(z_i)}{\tfrac{dQ_{n+1}}{dQ_1}(z_{n+1})}
                    \end{align*}
                    Substituting this into \eqref{eq:average_RN} and rewriting yields
                    \begin{align*}
                        \frac{dQ}{d\overline{Q}}(z^{n+1})
                            = \frac{\tfrac{dQ_{n+1}}{dQ_1}(z_{n+1})}{\frac{1}{n + 1} \sum_{i=1}^{n+1} \tfrac{dQ_{n+1}}{dQ_1}(z_i)}.
                    \end{align*}
                \end{proof}    
                \subsection{Proof of Proposition \ref{prp:minimax}}
                \begin{proof}
                    We have the implication $L_{Z_{n+1}}(\delta) > \sup_{z \in C_\alpha^{Z^n}} L_z(\delta) \implies Z_{n+1} \not\in C_\alpha^{Z^n}$, so that $ P\left(L_{Z_{n+1}}(\delta) > \sup_{z \in C_\alpha^{Z^n}} L_z(\delta)\right) \leq P\left(Z_{n+1} \not\in C_\alpha^{Z^n}\right) \leq \alpha$.
                \end{proof}
            \subsection{Proof of Proposition \ref{prop: data-dep}}
            \begin{proof}
                    We have $\sup_{z \in \mathcal{Z}} L_{z}(\delta) / \mathcal{E}^{Z^n}(z) \geq L_{Z_{n+1}}(\delta) / \mathcal{E}^{Z^n}(Z_{n+1})$, so that 
                    \begin{align*}
                        \frac{L_{Z_{n+1}}(\delta)}{\sup_{z \in \mathcal{Z}} L_{z}(\delta) / \mathcal{E}^{Z^n}(z)} \leq \frac{L_{Z_{n+1}}(\delta)}{L_{Z_{n+1}}(\delta) / \mathcal{E}^{Z^n}(Z_{n+1})} 
                        = \mathcal{E}^{Z^n}(Z_{n+1}).
                    \end{align*}
                    The result then follows from the validity of $\mathcal{E}^{Z^n}$.
                \end{proof}
                 \newpage

                \begin{figure*}[h!]
                    \centering
                    \includegraphics[width=0.48\textwidth]{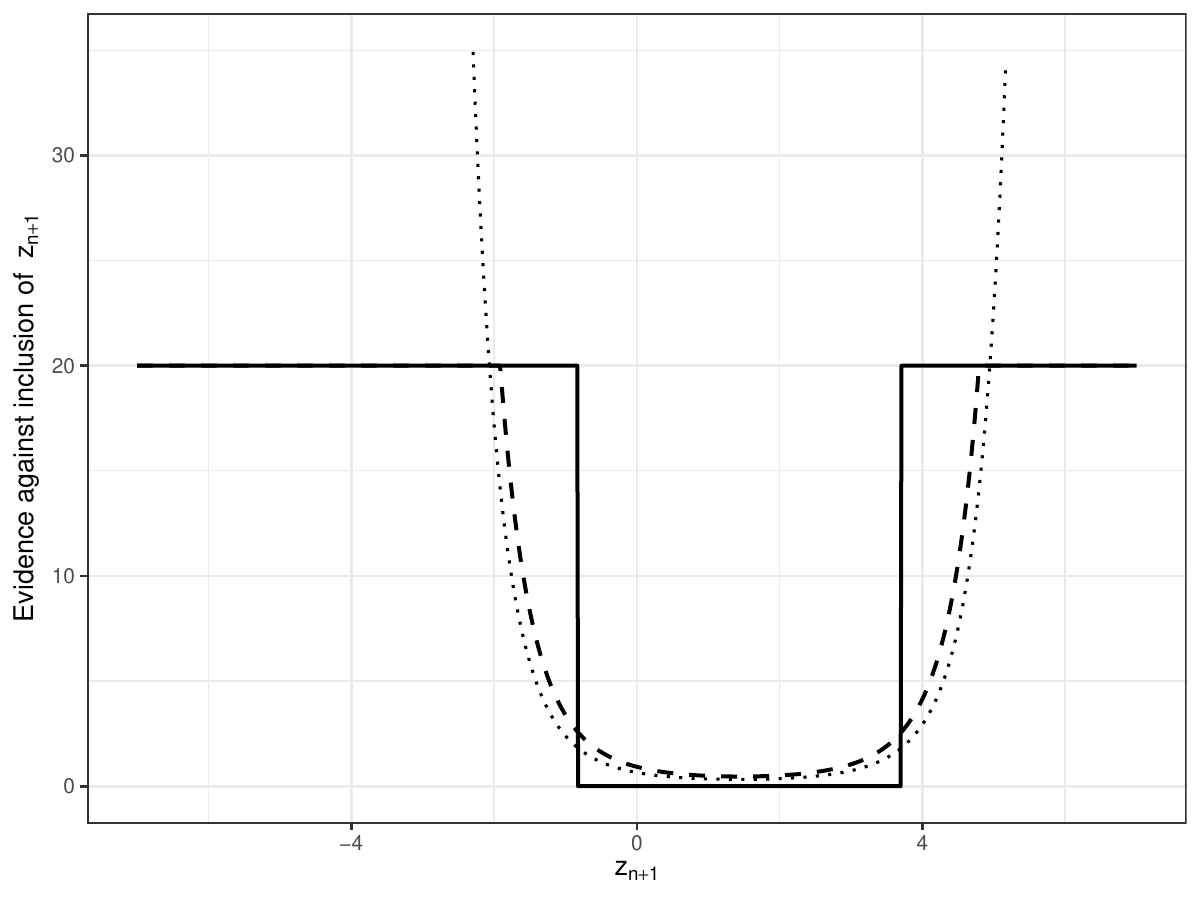}
                    \hfill
                    \includegraphics[width=0.48\textwidth]{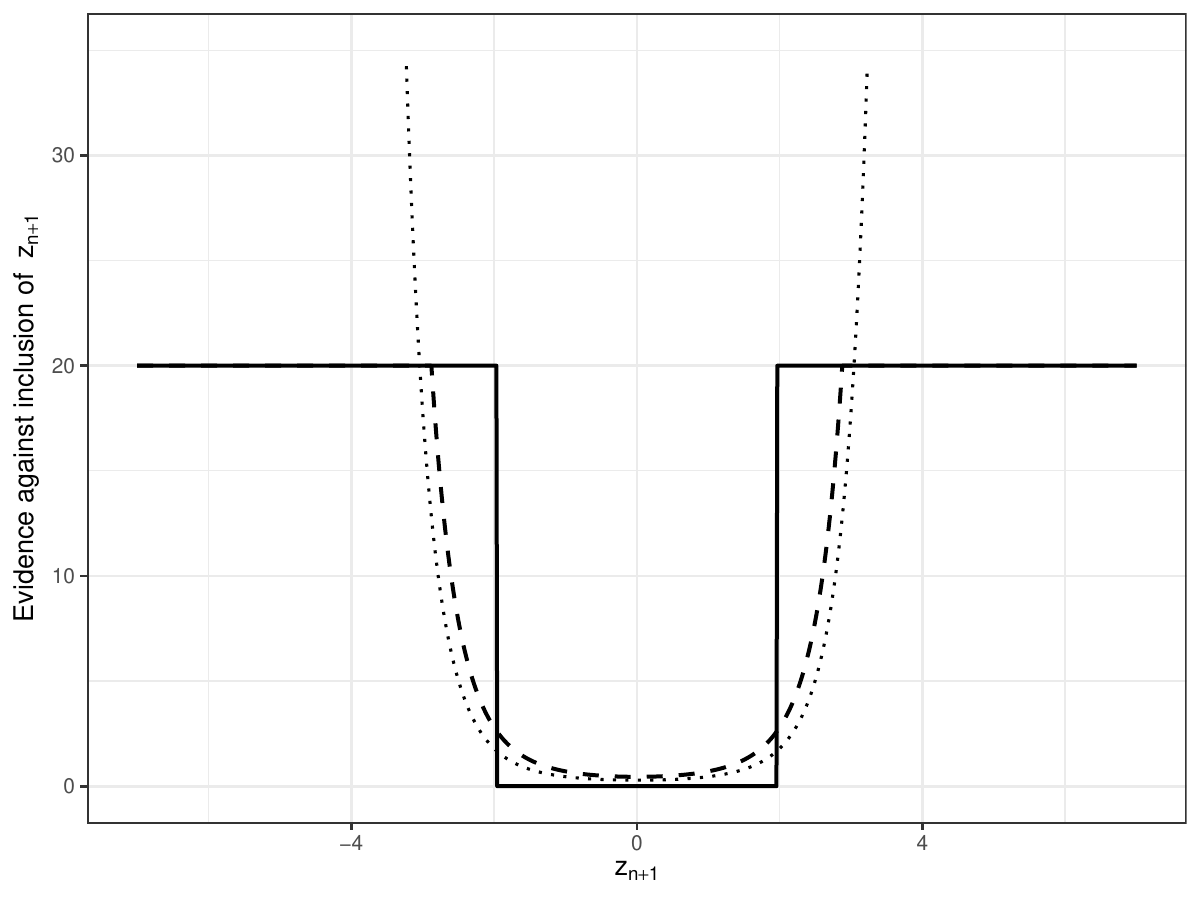}
                    \caption{Optimal fuzzy prediction sets for Neyman--Pearson-utility (solid), log-utility (dotted) and bounded log-utility (dashed) under the simple Gaussian setting (left) and composite Gaussian setting (right) from Example \ref{exm:simple_gaussian_NP}, \ref{exm:composite_gaussian_NP} and \ref{exm:fuzzy_gaussian}, for $n = 3$, $\mu = 0$, $\sigma = 1$, $\tau = 3.5$, $\bar{Z}_n = 1.44$. Here, the Neyman--Pearson utility and bounded log-utility are both bounded at $1/0.05 = 20$.}
                    \label{fig:gaussians}
                \end{figure*}
                \section{Example Fuzzy Gaussian}\label{sec:ex_fuzzy}
                \begin{exm}[Fuzzy Gaussian]\label{exm:fuzzy_gaussian}
                    In the examples of Section \ref{sec:examples}, we implicitly considered optimal fuzzy prediction sets under the classical Neyman--Pearson utility $x \mapsto x \wedge 1/\alpha$.
                    We continue these examples under more general utility functions.

                    We first consider the log-utility version of Example \ref{exm:simple_gaussian_NP}, where we had the simple model $\mathcal{P} = \{\mathcal{N}(\mu 1_{n+1}, \sigma^2 I_{n+1})\}$ and optimized against the alternative $Q = \mathcal{N}(\mu 1_{n}, \sigma^2 I_{n}) \times \mathcal{N}(\mu, \tau^2)$, for known $\mu \in \mathbb{R}$ and $\tau > \sigma > 0$.
                    Here, we could choose $\tau \to \infty$ to mimic the Lebesgue measure, as in the preceding examples, but we find that this generally yields undesirable fuzzy prediction sets.
                    
                    It is well-known that the log-utility-optimal e-value is the likelihood ratio itself.
                    Hence, the log-utility-optimal e-value here is simply the likelihood ratio
                    \begin{align}\label{eq:gaussian_LR_simple}
                        \mathcal{E}(z^{n+1})
                            = \frac{d\mathcal{N}(\mu, \tau^2)}{d\mathcal{N}(\mu, \sigma^2)}(z_{n+1})
                            = \mathcal{E}^{z^{n}}(z_{n+1}).
                    \end{align}
                    For $1/\alpha$-bounded log-utility $U : x \mapsto \log(\mathcal{E} \wedge 1/\alpha)$, the optimal e-value is the likelihood ratio capped at $1/\alpha$ and `boosted' by some constant $b_\alpha$ so that it has expectation exactly 1 \citep{koning2025continuoustestingunifyingtests}:
                    \begin{align*}
                        \mathcal{E}^{z^{n}}(z_{n+1})
                            = \left(b_\alpha \frac{d\mathcal{N}(\mu, \tau^2)}{d\mathcal{N}(\mu, \sigma^2)}(z_{n+1})\right) \wedge 1/\alpha.
                    \end{align*}
                    
                    In the left panel of Figure \ref{fig:gaussians}, we display these fuzzy prediction sets next to the traditional prediction set from Example \ref{exm:simple_gaussian_NP}.

                    For the composite setting introduced in Example \ref{exm:composite_gaussian_NP}, we had derived the likelihood ratio
                    \begin{align*}
                        \textnormal{LR}(A, B)
                            = \frac{d\mathcal{N}(0, \tau^2 + \sigma^2/n)}{d\mathcal{N}(0, \sigma^2 + \sigma^2/n)}(B),
                    \end{align*}
                    where $A = Z^n - \bar{Z}_n$ and $B = Z_{n+1} - \bar{Z}_n$.
                    Fixing $\bar{Z}_n$ and evaluating this likelihood ratio at plug-in values $z_{n+1} \in \mathbb{R}$ for $Z_{n+1}$, we recover the log-optimal fuzzy prediction set:
                    \begin{align}\label{eq:gaussian_LR_composite}
                        \mathcal{E}^{Z^n}(z_{n+1})
                            = \frac{d\mathcal{N}(\bar{Z}_n, \tau^2 + \sigma^2/n)}{d\mathcal{N}(\bar{Z}_n, \sigma^2 + \sigma^2/n)}(z_{n+1}),
                    \end{align}
                    which relies on a variant of the Hunt-Stein theorem derived for log-optimal e-values proven by \citet{perez2022statistics}.
                    To the best of our knowledge, this has not been proven yet for general expected-utility-optimal e-values, but we would be surprised if it does not go through in general.
                    
                    Comparing the fuzzy prediction set in \eqref{eq:gaussian_LR_composite} to \eqref{eq:gaussian_LR_simple}, we see that this likelihood ratio is centered at $\bar{Z}_n$ instead of $\mu$ and the variance in both numerator and denominator is inflated by an additional $\sigma^2/n$-term due to the estimation of $\mu$ by $\bar{Z}_n$.
                    We illustrate this fuzzy prediction set in the right panel of Figure \ref{fig:gaussians}, along its bounded version and the Neyman--Pearson-utility variant from Example \ref{exm:composite_gaussian_NP}.
                \end{exm}
                
            \section{Clipping, capping and dampening}\label{sec:clip}
                In this section, we cover some practical tools for designing utility functions.
            
                For subsequent decision making, as treated in Section \ref{sec:certified_decisions}, it may be desirable to place a lower bound $0 < b \leq 1$ on our e-values, such as $b = 0.01$ or $b = 0.1$. 
                This prevents settings in which a single zero-valued e-value may come to dominate subsequent decision making due to the minimax nature of the decisions.
                The log-optimal e-value with such a lower bound is
                \begin{align}\label{eq:clip}
                    \mathcal{E}(z^{n+1}) = \left(\lambda \textnormal{LR}^{Z_{n+1}|O(z^{n+1})}(z_{n+1})\right) \vee b,
                \end{align}
                where $\lambda$ is some normalization constant that ensures $\mathcal{E}$ is an exact e-value.
                This may be interpreted as `clipping' the likelihood ratio from below, and simultaneously shrinking it by $\lambda$ to ensure it remains a valid e-value.
                Alternatively, it can be viewed as choosing a utility function with $U(x) = -\infty$ on $[0, b)$.

                Analogously, in case we are not interested in evidence beyond some value $c \geq 0$, we may also \emph{cap} our e-values below some upper bound $c \geq 1$.
                The resulting optimizer is analogous to \eqref{eq:clip}, with $\vee b$ is replaced by $\wedge c$ \citep{koning2025continuoustestingunifyingtests}.
                Capping may also be incorporated in the utility function, and appears in the Neyman--Pearson utility function $U(x) = x \wedge 1/\alpha$.

                An alternative strategy to `dampen' e-values is introduced by \citet{grunwald2023posterior}.
                He proposes to take an e-value $\mathcal{E}$ and dampen it to the e-value $b + (1 - b)\mathcal{E}$ that may be interpreted as a $b$-mixture of the e-value and the constant 1.
                If $\mathcal{E}$ is the log-utility maximizer, then we find $b + (1 - b)\mathcal{E}$ implicitly maximizes the utility function $U(\mathcal{E}) = \log([\mathcal{E} - b] \vee 0)$.       
                
         \section{Connecting as-if and weighted decisions}\label{app:connection}
                While the weighted loss approach for fuzzy prediction sets in Section \ref {sec:weighted_loss} and the as-if decision approach from Section \ref{sec:as-if} may seem distinct, we show that the as-if approach may be viewed as a special limiting case of the weighted loss approach.
                This relationship was not yet established before; \citet{andrews2025certified} describe both approaches as distinct paradigms.
                The connection is not obvious, as it relies on our observation that the E-posterior coined by \citet{grunwald2023posterior} is equivalent to a fuzzy confidence set, and so a generalization of a confidence set.
    
                Let $\gamma \in (0, 1)$, and define the following fuzzy prediction set for every $z_{n+1} \in \mathcal{Z}$:
                \begin{equation}
                    \mathcal{E}^{Z^n}(z_{n+1})
                        := \gamma \frac{1}{\alpha}\mathbb{I}\{z_{n+1} \notin \mathcal{C}_\alpha^{Z^n}\} + (1-\gamma) \mathbb{I}\{z_{n+1} \in \mathcal{C}_\alpha^{Z^n}\}.
                \end{equation}
                Note that $\mathcal{E}$ converges to a non-fuzzy prediction set when $\gamma \to 1$.
                Moreover, this is indeed an e-value because it is dominated by the mixture $\gamma\frac{1}{\alpha}\mathbb{I}\{Z_{n+1}\notin \mathcal{C}_\alpha^{Z^n}\} + (1-\gamma)$ of a valid e-value $\frac{1}{\alpha}\mathbb{I}\{Z_{n+1} \notin \mathcal{C}_\alpha^{Z^n}\}$ and the constant 1, and a mixture of valid e-values is itself a valid e-value.
    
                Plugging this into definition of the risk bound for a given decision $d$ yields
                \begin{align*}
                    &\sup_{z} \frac{L(d, z)}{\mathcal{E}^{Z^n}(z)} 
                        = \sup_{z} \frac{L(d, z)}{\gamma \frac{1}{\alpha}\mathbb{I}\{z \notin \mathcal{C}_\alpha^{Z^n}\}+ (1-\gamma) \mathbb{I}\{z \in \mathcal{C}_\alpha^{Z^n}\}} \\
                        &= \sup_{z} \left\{\frac{\alpha}{\gamma}\mathbb{I}\{z \not\in \mathcal{C}_\alpha^{Z^n}\} L(d,z) + \frac{1}{1-\gamma}\mathbb{I}\{z \in \mathcal{C}_\alpha^{Z^n}\} L(d,z)\right\}.
                \end{align*}
    
                Now, assuming $\mathcal{C}_\alpha^{Z^n}$ is non-empty and $L(d, z)$ is bounded, we have for $\gamma$ sufficiently close to 1 that this equals
                \begin{align*}
                     \sup_{z} \frac{1}{1-\gamma}\mathbb{I}\{z \in \mathcal{C}_\alpha^{Z^n}\} L(d,z)
                        = \frac{1}{1-\gamma} \sup_{z \in \mathcal{C}_\alpha^{Z^n}} L(d, z),
                \end{align*}
                which is proportional to the risk function we minimize for non-fuzzy prediction sets in \eqref{eq:minimax}.
                This shows that the non-fuzzy `as-if' decision can indeed be viewed as a limiting case of the fuzzy procedure.
                
     \section{Covariates}\label{sec:covariates}
            In many applications, the observations $Z^{n+1}$ are decomposed into $Z^{n+1} = (X^{n+1}, Y^{n+1})$, where the outcomes $Y^{n+1}$ come with covariates $X^{n+1}$.
            In such a setting, the covariate $X_{n+1}$ is observed alongside $Z^n = (X^n, Y^n)$, and we are to construct a prediction set for $Y_{n+1}$.
            In this section, we show how such covariates are easily incorporated into our framework.
            We only cover the non-fuzzy setting, as fuzzy prediction sets may be derived analogously.

            When adding covariates, our prediction set $C_\alpha^{(Z^n, X_{n+1})}$ now also depends on $X_{n+1}$.
            The coverage guarantee is also marginal over the covariates:
            \begin{align*}
                P(Y_{n+1} \in C_\alpha^{(Z^n, X_{n+1})}) \geq 1 - \alpha,
            \end{align*}
            for every $P \in \mathcal{P}$.

            Following the discussion in Section \ref{sec:slicing}, we may still construct a prediction set $C_\alpha$ for $Z^{n+1}$, and subsequently slice it at the realization $(Z^n, X_{n+1}) = (z^n, x_{n+1})$ to obtain our prediction set $C_\alpha^{(z^n, x_{n+1})}$ for $Y_{n+1}$:
            \begin{align*}
                C_\alpha^{(z^n, x_{n+1})}
                    := \{y \in \mathcal{Y} : C_\alpha(z^n, (x_{n+1}, y)) = 0\}.
            \end{align*}
            
            The optimality results remain the same, but we now maximize the power under the distribution
            \begin{align*}
                Q = P_*^{Z^n, X_{n+1}} \otimes \mu^{|(Z^n, X_{n+1})}.
            \end{align*}

        \subsection{Conformal prediction with covariates}\label{sec:conformal_covariates}
            In the application to the conformal setting, the only thing that changes is that we may use some additional structure in $Q$.
            In particular, we must replace $Q^{Z_{n+1}} = \mu^{|Z^n}$ by $Q^{Z_{n+1}} = P_*^{X_{n+1}|Z^n} \otimes \mu^{|(X_{n+1}, Z^n)}$

            For example, the i.i.d. result in Proposition \ref{prp:conditional_LR_indep} now specializes to
            \begin{align*}
                \textnormal{LR}^{|O(z^{n+1})}(z^{n+1})
                    &\propto \frac{dQ^{Z_{n+1}}}{dQ^{Z_1}}(z^{n+1})
                    = \frac{d(P_*^{X_{n+1}} \otimes \mu^{|X_{n+1}})}{dP_*^{X_{n+1}, Y_{n+1}}}(z_{n+1})
                    = \frac{dP_*^{X_{n+1}} \otimes d\mu^{|X_{n+1}}}{dP_*^{X_{n+1}} \otimes  dP_*^{Y_{n+1}|X_{n+1}}}(z_{n+1}) \\
                    &= \frac{d\mu^{|X_{n+1}}}{dP_*^{Y_{n+1}|X_{n+1}}}(y_{n+1}),
            \end{align*}
            where the normalization constant is
            \begin{align*}
                \frac{1}{(n+1)} \sum_{i = 1}^{n+1} \frac{d\mu^{|X_{i}}}{dP_*^{Y_{n+1}|X_{i}}}(y_i).
            \end{align*}
            Here, we may write $P_*^{Y_{n+1}|X_{i}}$ instead of $P_*^{Y_{i}|X_{i}}$ due to the i.i.d. assumption.

            To apply this, we may plug-in some estimator $\widehat{P}^{Y_{n+1}|X_{n+1}}$, for the kernel $P_*^{Y_{n+1}|X_{n+1}}$, or several estimators and average the resulting e-values.
            We use this in Section \ref{sec:application}, where we estimate this kernel through a neural network.
      
                \end{document}